\documentclass[10pt,a4paper]{article}

\usepackage[utf8]{inputenc}
\usepackage[left=3cm, right=2.5cm, top =2.5cm, bottom = 2.5cm]{geometry}

\usepackage[ngerman, english]{babel}
\usepackage[utf8]{inputenc}	
\usepackage[T1]{fontenc}	
\usepackage{csquotes} 

\usepackage{pdflscape} 

\usepackage{hyperref}

\usepackage{algorithm}
\usepackage{algpseudocode}

\usepackage{enumitem}

\usepackage{graphicx}
\usepackage[export]{adjustbox}
\usepackage{subcaption} 
\usepackage{rotating} 

\usepackage{verbatim}

\usepackage{amsmath,	
	amssymb,			
	amsthm,				
	bbm, 				
	mathtools} 			
\usepackage{mathrsfs}
\usepackage{dsfont} 

\usepackage{color}

\usepackage{scrlayer-scrpage}

\usepackage{cancel}

\usepackage{multirow}
\usepackage[backend=biber, style = alphabetic]{biblatex}

\usepackage{todonotes}

\usepackage{tikz-cd}
\usetikzlibrary{intersections}
\usetikzlibrary{babel}

\newcommand{\Z}{\mathbb Z}
\newcommand{\N}{\mathbb N}

\newcommand{\R}{\mathbb R}

\newcommand{\uprd}{\,\mathrm{d}}

\newcommand{\abs}[1]{\left\lvert#1\right\rvert}
\newcommand{\norm}[1]{\left\lVert#1\right\rVert}

\newcommand{\SP}[2]{\left\langle #1 , #2\right\rangle }
\newcommand{\set}[2]{\left\lbrace#1 \vphantom{#2}\right|\left. \vphantom{#1}#2\right\rbrace}

\newcommand\restr[2]{{
  \left.\kern-\nulldelimiterspace 
  #1 
  \vphantom{\big|} 
  \right|_{#2} 
  }}

\newtheoremstyle{break}
  {\topsep}{\topsep}%
  {\itshape}{}%
  {\bfseries}{}%
  {\newline}{}%
\theoremstyle{break}
\newtheorem{thm}{Theorem}

\newtheorem{bem}[thm]{Remark}

\newtheorem{lem}[thm]{Lemma}

\newtheorem{ass}[thm]{Assumption}
\newtheorem{example}{Example}[]

\newcommand{\nameMethode}{double discrepancy principle}
\newcommand{\nameRegmethode}{double regularisation}

\title{Regularisation and Parameter Choice Rules for Composite Linear Inverse Problems}

\author{Christina Brandt$^1$
, Patrick Horn$^1$
}
\date{
\scriptsize $^1$ University of Greifswald\\
Institute for Mathematics \& Computer Science\\
Walter-Rathenau-Straße 47, 17487 Greifswald\\
}
\begin{document}

\maketitle

\begin{abstract}
\noindent
In this paper we propose a novel regularisation strategy for structured linear inverse problems of the form $Ax=F \circ G x = z$. Instead of using the operator $A$ or a complex scheme involving both $F$ and $G$ simultaneously, combining already existing and simpler regularisations for $F$ and $G$ allows the use of established and optimised programs for high-speed computations. As classical parameter choice rules fail in the scenario when no source conditions are known because intermediate error bounds may be unavailable, we provide both a-priori and a-posteriori rules with a convergence speed analysis given source conditions.
We test the proposed method on the pre-clinical medical imaging technique of magnetic particle imaging with different choices of regularisation 
and on a convolution problem for an analysis of the convergence speed.
\end{abstract}

\section{Introduction}
Considering a linear inverse problem, we are interested in solving the operator equation
\begin{equation}\label{eqn:OpEq}
Ax = z,
\end{equation}
with a linear map $A : X \to Z$ between  Hilbert spaces $X$ and $Z$. However, the right hand side $z$ is not readily available
but only 
a noisy version $z^\delta$ with $\norm{z-z^\delta}\leq \delta$ with 
noise level  $\delta$. The operator is often times ill-posed and  naive solutions are thus  heavily corrupted by the noise. Regularisation methods are necessary to stabilise the inversion process. Therefore, the reconstruction error  is split in two parts resulting in 
	the data error 
	and the approximation error. 
It is  necessary to balance both types of error against each other. 
Classical regularisation approaches include variational  or iterative methods. Variational methods are based on minimisation  of functionals  consisting of   
	a data term which ensures that the minimising solution solves the operator equation \eqref{eqn:OpEq} approximately and a regularisation term penalising undesired behaviour such as oscillations or enforcing desired behaviour like sharp edges, depending on the application. Machine learning can be included in modern approaches, see for example \cite{Arridge+_2019} for an overview and further \cite{ZhangPNP, HuraultChambollePNP, Cannas+_2026, RomanoRED}. Regularisation methods include a parameter describing how 
	well the unstable inverse of $A$ is approximated and which therefore influence data  and approximation error. For high-quality reconstructions it is not only important to choose an appropriate regularisation, but also to pick the regularisation parameter properly in dependence of the noise level. Additional information about the forward operator can be used to choose or design the regularisation, for example by including PDE constraints on the solution \cite{Westen+_2024, Abdulla+_2025, Moon+_2026}.

In this manuscript, we assume that the operator $A$
\begin{equation}\label{eqn:OpSplitting}
A = G \circ F
\end{equation}
is composed of two linear maps  $F: X \to Y$ and $G : Y \to Z$ with an additional Hilbert space $Y$. This type of additional information was already used to design regularisation terms, see for example \cite{Westen+_2024} or \cite{Burger_2023}. Instead of applying one adapted but complicated regularisation scheme, we want to split the regularisation process into multiple regularisation steps which may be easier to compute, i.e. we want to design regularisations for the operator equation~\eqref{eqn:OpEq} using regularisations for the operators $F$ and $G$.

As already mentioned above, not only the choice of regularisation is important, but also the way to choose the regularisation parameter. For the most extreme cases, where no error bound exists,  heuristic parameter choice rules such as the L-curve method \cite{HansenOleary_1993, KindermannRaik_2020} or the Hanke-Raus rule \cite{HankeRaus_1996, Jin_2016} can be used. Even though they perform well in practise, as they do not depend on the noise level, they don't stabilise the reconstruction in the classical sense and cannot form a regularisation scheme. Nevertheless, heuristic methods can be analysed 
and  convergence to the ground truth can be guaranteed \cite{Kindermann_2011, KindermannRaik_2019, Bakushinsky_1984}. On the other hand, given an error bound and additional smoothness information about the ground truth, it is  possible to design parameter choice rules behaving optimally. 
Without this smoothness information on the ground truth, it is still possible to use a-posteriori  parameter choice rules as the Morozov discrepancy principle \cite{Morozov_1968} which then performs suboptimally. 

In this study, we extend  classical results based on filters to the presented case with a composite operator. We will first show that under mild assumptions, given smoothness information,  classical parameter choice rules again result in optimal methods. Moreover, we analyse the situation without smoothness information. Here, it is not possible to directly use the Morozov method as for the second operator equation no error estimate is known. We propose an adapted discrepancy principle and show the order of convergence.

We use 
magnetic particle imaging (MPI) as a test application to evaluate our approach. MPI is a novel, pre-clinical medical imaging modality introduced in 2005  \cite{GleichWeizenecker_2005}. Here, a magnetisable tracer is injected in the blood flow of a patient which reacts to an applied dynamic magnetic field. Its non-linear response is measured allowing reconstruction of the spatial distribution of the tracer.
Naturally, different settings in the magnetic field result in different forward operators. Applying particular 
magnetic fields, the forward operator is given by a composition of convolution $G$ and  Radon transform $F$  \cite{Knopp_2011}, which was used in \cite{Weinmann+_2025} to construct a regularisation scheme.

Of course, the operator splitting is not limited to MPI.  
Another application which is given as a composite operator is the image zooming operator, i.e. the change of resolution in a picture. It can be interpreted as a concatenation of a blurring operator and a subsample operator \cite{EbrahimiVrscay_2007}. In medical imaging, it is sometimes desired to measure less data, for example if the measuring process is destructive in some sense. The problem can be interpreted as the usual operator $F$  followed by an operator $G$ forgetting some part of the measurement. Examples include compressed sensing in MRI \cite{YeMRICS}, limited angle tomography \cite{Tovey+_2019} and low-count PET \cite{WangQi_2016}.
Other pre-processing steps as artefact removal in the measurement-space \cite{Lequan+_2021} or noise-reduction \cite{HebertGopal_1994} fit this splitting as well. This corresponds to a concatenation of the measurement-operator $F$ and the identity operator $G = \mathrm{id}$.
Depending on the application, not only  a reconstructed tomographic image  but also  a segmentation is sought. This means, we want to reconstruct a solution $x^\dagger \in \tilde X$ solving the inverse problem \eqref{eqn:OpEq}, where $\tilde X \subset X$. For example, if $X$ denotes a space of pixelated images, $\tilde X$ could be  the corresponding space where each pixel value has to be equal to some pre-defined values. In total, we can interpret this as a inverse problem of the above form, where $F$ is the embedding $\tilde X \hookrightarrow  X$ and $G$ is the measurement operator, see for example \cite{UNetRonneberger+_2015} or  \cite{KuopioChallenge2023}.
General post-processing can be applied to blurry images as in \cite{HamiltonHauptmann_2018}.

In this manuscript, we investigate  regularisations tailored to the above-mentioned structure of the operator equation. First, we present sufficient conditions such that the composition of regularisations possesses the regularisation property in Subsection~\ref{subsec:RegProperty}. Afterwards in Subsection~\ref{sec:convSpeed}, we analyse the convergence speed of the method given source conditions for two cases. 
We investigate the case where only one operator $F$ or $G$ is ill-posed and  the case where both operators are ill-posed, proposing a-priori and a-posteriori parameter choice rules.
Section~\ref{sec:NumExp} consists of numerical experiments testing qualitative and quantitative performance, once on simulated MPI data (Subsection~\ref{sec:NumerikMPI}) and on a deconvolution task (Subsection~\ref{sec:NumerikConv}). We conclude our findings in Section~\ref{sec:Conclusion}.

\section{Regularisation of Composite Operator Equations}\label{sec:RegComp}

In the following, all occurring spaces are Hilbert spaces if not otherwise mentioned. In the case of an ill-posed operator $A: X \to Z$ in the operator equation~\eqref{eqn:OpEq}, solutions may not exist or may not be unique. Nevertheless, to achieve some reconstruction, a minimum-norm solution $x^\dagger$ to the least-squares problem $x^\dagger = \min_x \norm{Ax-z}^2$ can be computed with help of the Moore-Penrose inverse $A^\dagger$, i.e. the solution is given by $A^\dagger z$. In practise, the true data $z$ is not available but only a noisy version $z^\delta$ with $\norm{z-z^\delta}<\delta$. Thus, in the noise-free case, a reconstruction via $x^\dagger := A^\dagger z$ is possible. In the case of noise, this formula may not be applied, as the Moore-Penrose inverse is unbounded for ill-posed operator equations. 
Instead, regularisation is necessary for the stabilisation of the inversion.  Given the ill-posed operator equation \eqref{eqn:OpEq}, a map  $\alpha: \R_{> 0}\times Z \to \R_{> 0}$ and a family of operators $\left\{ R_\alpha \right\}_{\alpha}$ , the tuple $(\{R_\alpha\}_\alpha, \alpha)$ is said to be a regularisation for the pseudo-inverse  $A^\dagger$ with parameter choice rule $\alpha$ if 
\begin{enumerate}
	\item $\sup \set{ \norm{R_{\alpha(\delta, z^\delta)}z^\delta - A^\dagger z} }{ z^\delta \in Z, \norm{z - z^\delta}\leq \delta}\rightarrow 0$ as $\delta \searrow 0$,

\item	$\sup \set{ \alpha(\delta, z^\delta) }{ z^\delta \in Z, \norm{z-z^\delta} \leq \delta}\rightarrow 0$ as $\delta \searrow 0$.
\end{enumerate}
See definition 3.1 in  \cite{EHN_1996}.

\subsection{Regularisation Property of the double regularisation}\label{subsec:RegProperty}
We consider the forward operator $A$ of the form
	\[A = G \circ F\]
	with $F: X \to Y$ and $G : Y \to Z$. Assuming all operators to be invertible  results in 
	$A^{-1} = F^{-1} \circ G ^{-1}$. Introducing regularisations $R^G_\alpha$ and $R^F_\gamma$ for $G^\dagger$ and $F^\dagger$ as approximations to the inverse operators $G^{-1}$ and $F^{-1}$ respectively, motivates the regularisation approach $R^A = R^F_\gamma \circ R^G_\alpha$, which converges to the concatenation of the inverses, i.e. $R^F_\gamma \circ R^G_\alpha \to F^{-1} \circ G^{-1}$ for $\gamma, \alpha \to 0$. However, in general with non-invertible $F$ and $G$,  we have $A^\dagger \neq F^\dagger \circ G^\dagger$ and we cannot expect any convergence to the pseudo-inverse of $A$. Sufficient and necessary conditions for $A^\dagger = F^\dagger \circ G^\dagger$ are given in \cite{Ben-Israel_2003, Greville_1966}.
We thus assume $A^\dagger = F^\dagger \circ G ^\dagger$  to obtain the following result.

\begin{thm}\label{thm:compRegIsReg}
	Let $A$, $F$ and $G$ be operators between Hilbert spaces $X$, $Y$ and $Z$ as defined above. Moreover, let $\mathfrak Z \subset Z$ be a subset of $Z$ such that for all $z\in \mathfrak Z \subset Z$ there exists a minimum-norm solution $y^\dagger\in Y$ to the operator equation $Gy =z$ and a minimum-norm solution $x^\dagger \in X$ for the operator equation $Fx = y^\dagger$.
	Assume there exists a parameter choice rules $\gamma: \R_{>0} \times Z \to R_{>0}$ 
	with 
	\[\sup \set{\gamma(\delta, z^\delta)}{\norm{z-z^\delta}\leq \delta} \overset{\delta \to 0}{\longrightarrow} 0\]
	and a parameter choice rule $\alpha: \R_{>0} \times Z \to \R_{>0}$ such that
	\begin{equation}\label{eqn:CondcompReg}
		\sup \set{ \norm{R^F_{\gamma(\delta, z^\delta)}}\norm{y^\dagger - R^G_{\alpha(\delta, z^\delta)}z^\delta}_Y }{ \norm{z-z^\delta}\leq \delta} \overset{\delta \to 0}{\longrightarrow} 0\,.
	\end{equation}
	Then the concatenation $R^F \circ R^G$ is a regularisation for  $\restr{A^\dagger}{\mathfrak Z}$, i.e.
	\[\sup \set{ \norm{x^\dagger - R^F_{\gamma(\delta, z^\delta)} R^G_{\alpha(\delta, z^\delta)}z^\delta}_X }{ \norm{z-z^\delta}\leq \delta, z \in \mathfrak Z} \overset{\delta \to 0}{\longrightarrow} 0\,.\]
	\begin{proof}
		In the following, we will abbreviate $\gamma(\delta, z^\delta)$ with $\gamma$. As $R_\gamma^F$ defines a regularisation for $F^\dagger$ and $y^\dagger$ denotes the noise-free right hand side in $Fx = y^\dagger$,
		it holds \[\norm{x^\dagger - R^F_\gamma y^\dagger}_X\overset{\gamma \to 0}{\longrightarrow} 0\,.\]   
		Further, by the regularisation property of $R^G$
		\[\norm{y^\dagger - R^G_{\alpha(\delta, z^\delta)}z^\delta}_Y \overset{\delta \to 0}{\longrightarrow} 0\]
		holds. The   
		assumption on $z \in \mathfrak Z$  implies 
		\begin{align*}
			\norm{x^\dagger - R^F_\gamma R^G_{\alpha(\delta, z^\delta)}z^\delta}_X&\leq \norm{x^\dagger - R^F_\gamma y^\dagger}_X  + \norm{R^F_\gamma\left( y^\dagger - R^G_{\alpha(\delta, z^\delta)}z^\delta\right) }_X\\
			&\leq \norm{x^\dagger - R^F_\gamma y^\dagger}_X + \norm{R^F_\gamma}\norm{y^\dagger - R^G_{\alpha(\delta, z^\delta)}z^\delta}_Y \overset{\delta \to 0}{\longrightarrow} 0\,.
		\end{align*}
		This shows the claim.
	\end{proof}
\end{thm}
We denote this kind of regularisation scheme in the following \textbf{\nameRegmethode}.
We assumed that for $\delta \to 0$  the convergence of $R^Gz^\delta$ to $G^\dagger z$ is faster than the divergence of $\norm{R^F_\gamma}$. For this, the restriction of the exact data $z$ to a subset $\mathfrak Z$ is crucial. In general, for any ill-posed operator $G$ there cannot exist such a convergence speed result for  arbitrary $z \in Z$, see  \cite{Bakushinsky_1984, EHN_1996}.
The needed type of results are however readily available for restricted true solutions, see \cite[Section 3.2]{EHN_1996}. Based on these assumptions, we will analyse the convergence speed of the DDP in Section~\ref{sec:convSpeed}.

\subsection{Analysis of the Convergence Speed}\label{sec:convSpeed}
In this section, we analyse the convergence speed of \nameRegmethode ~approaches given source conditions. Assume in the following that $F$ and $G$ are compact and simultaneously diagonalisable, i.e.
let
\begin{equation}\label{eqn:SVD_FG}
	Fx = \sum _n \lambda _n \SP{x}{x_n}y_n, \quad Gy = \sum_n \mu_n \SP{y}{y_n}z_n.
\end{equation}

It is not necessary to assume both operators to be compact but this simplifies the analysis. As a result of \eqref{eqn:SVD_FG}, the composed operator can be represented as 
$Ax = \sum _n \lambda_n \mu_n \SP{x}{x_n}z_n$ and its pseudo-inverse is given by $A^\dagger = F^\dagger G^\dagger$.

As outlined in the preceding discussion
, it is impossible to achieve a constructive result concerning convergence speed. However, restricting
the true solution  $x^\dagger \in X_\nu^A \subset X$ to subspaces of $X$, 
it is possible to derive convergence speed results and even to achieve 
 optimal convergence speed.
For a given operator $A:X \to Z$, we define the space 
\[
X_\nu^A = \left\{ x \in \text{dom}(A) \mid \exists \tilde x \in \mathcal N(A)^\perp: x = \abs{A}^\nu \tilde x \right\}, \quad \nu > 0\,.
\]
To analyse this type of spaces, we use the Picard-condition, see Theorem~\ref{thm:Picard} in the appendix.

Consider  the case $x \in X_\nu ^A$, $\nu>0$. Further assume that $\lambda_i, \mu_i \leq 1$ for all $i \in \N$. We then know that 
\[\infty > \sum_n \frac{\abs{\SP{x}{x_n}}^2}{\lambda_n^{2\nu}\mu_n^{2\nu}}\geq \sum_n \frac{\abs{\SP{x}{x_n}}^2}{\lambda_n^{2\nu}}\]
and thus $x\in X_\nu^F$.
Moreover, we have $Fx \in X_\nu^G$ as
\[
\infty > \sum_n \frac{\abs{\SP{x}{x_n}}^2}{\lambda_n^{2\nu}\mu_n^{2\nu}} \geq \sum_n \frac{\lambda_n^2\abs{\SP{x}{x_n}}^2}{\mu_n^{2\nu}} = \sum_n \frac{\abs{\SP{Fx}{y_n}}^2}{\mu_n^{2\nu}}.
\]

In this case, it is possible to show the following lower limit for the reconstruction error.

\begin{thm}[\text{\cite[p. 55-59]{EHN_1996}}]
	For any compact A and a regularisation $(R_\alpha, \alpha)$, there exists a sequence $\delta_n \to 0$ such that
	\begin{equation}\label{eqn:OrderOptBound}
		\inf_\alpha \sup \set{ \norm{R_\alpha z^\delta - x^\dagger} }{ x^\dagger = \abs A ^\nu w, \, \norm{w}\leq \rho, \, \norm{Ax^\dagger - z^\delta}\leq \delta_n} \geq \delta_n^\frac{\nu}{\nu+1} \rho 	^\frac{1}{\nu+1}.
	\end{equation}
\end{thm}
Any regularisation fulfilling \eqref{eqn:OrderOptBound} with equality is called an optimal method. If  the left hand side can be bounded with $c\delta^\frac{\nu}{\nu+1} \rho 	^\frac{1}{\nu+1}$ it is called  order optimal.
\subsubsection{Trivial Cases}
We start the convergence analysis with the cases where one of the operators $F$ and $G$ is well-posed.
\subsubsection*{Case 1: F well-posed}
Consider the case where $F$ is well-posed, i.e. the pseudo-inverse of $F$ is continuous. Consequently,  we can choose $R^F_\alpha = F^\dagger$. Given the source condition $x^\dagger \in X_\nu^A$, we  have $Fx^\dagger = y^\dagger \in X_\nu^G$.  
An order optimal regularisation $(R^G_\alpha, \alpha)$ for $G^\dagger$ with parameter choice rule $\alpha=\alpha \left(\delta, z^{\delta}  \right)$ 
implies 
\begin{equation}\label{eqn:F_well_posed}
	\norm{x^\dagger - F^\dagger R^G_\alpha z^\delta} \leq \norm{x^\dagger - F^\dagger y^\dagger} + \norm{F^\dagger}\norm{y^\dagger - R^G_\alpha z^\delta}= \mathcal O\left(\delta^\frac{\nu}{\nu + 1}\right)
\end{equation}
and the regularisation is order optimal, too. In this case, each parameter choice rule of a regularisation fulfils the conditions stated in Theorem~\ref{thm:compRegIsReg}. Moreover, given $F = \mathrm{id}$ and an optimal regularisation $(R^G_\alpha, \alpha)$, the method is optimal, too.

The pseudo-inverse is not always available for use. Consider therefore any regularisation $(R^F_\gamma, \gamma)$ for $F^\dagger$ with an a-priori parameter choice rule $\gamma$ and the arbitrary but fixed sequences $\delta_n \to 0$ and $z^\delta_n$ such that $\norm{z-z^\delta_n} \leq \delta_n$. By the regularisation property, we have for each $y\in Y$ that $R^F_{\gamma(\delta_n)} y \longrightarrow F^\dagger y$ as $\gamma(\delta_n) \to 0$. For each $y \in Y$, we can bound the sequence
\begin{align*}
\norm{R^F_{\gamma(\delta_n)}y} \leq \norm{R^F_{\gamma(\delta_n)}y - F^\dagger y} + \norm{F^\dagger}\norm{y} \to \norm{F^\dagger}\norm{y}.
\end{align*}
Therefore, by the Banach-Steinhaus theorem, it follows that $\sup_n \norm{R^F_{\gamma(\delta_n)}} < \infty$. This yields with the same computation as in \eqref{eqn:F_well_posed} that
\[\norm{x^\dagger - R^F_{\gamma} R^G_\alpha z^\delta_n} = \mathcal O\left(\delta^{\frac \nu {\nu +1}}\right).\]

\subsubsection*{Case 2: G well-posed}

In the case, where $G$ is well-posed, $G^\dagger$ is continuous. Choosing $R^G_\alpha = G^\dagger$ and an order optimal regularisation $(R^F_\gamma, \gamma)$ for $F^\dagger$ with a-priori parameter choice rule 
$\gamma \equiv \gamma\left(\norm{G^\dagger}\delta\right)$, 
we obtain 
\begin{align*}
	\norm{x^\dagger - R^F_{\gamma} G^\dagger z^\delta} &\leq \norm{x^\dagger - R^F_\gamma y^\dagger} + \norm{R^F_\gamma\left(y^\dagger - G^\dagger z^\delta\right)}\\
	&\leq C\left(\norm{G^\dagger}\delta\right)^\frac{\nu}{\nu+1} = \mathcal O\left(\delta^\frac{\nu}{\nu+1}\right).
\end{align*}
As $\gamma$ is independent of $z^\delta$, we have $\norm{R^F_\gamma}\delta \to 0$ for $\delta \to 0$ and the condition~\eqref{eqn:CondcompReg} is fulfilled for the given choices of regularisations. Overall, we achieve an order optimal regularisation.

\subsubsection{Non-trival Case: both F and G ill-posed}
Assuming that $F$ and $G$ are ill-posed, we derive  a-priori and an a-posteriori rule.
\subsubsection*{A-Priori Parameter Choice Rule}
To do our analysis
, we need to take assumptions on the regularisation $R^G_\alpha$ and on the behaviour of the operator norm $\norm{R^F_\gamma}$ with regards to $\gamma$.
In the following, we assume $R_\alpha ^G$ to be an order-optimal regularisation, i.e. 
\[\norm{y^\dagger - R^G_\alpha z^\delta} = \mathcal O\left(\delta^\frac{\nu}{\nu + 1}\right).\]

We apply an a-priori parameter choice rule $\gamma \equiv \gamma(\delta)$ for the regularisation  $R^F_\gamma$ and suppose that the approximation error fulfil
\begin{equation}
	\norm{x^\dagger - R^F_\gamma y^\dagger} = \mathcal O\left(\gamma(\delta)^\frac \nu 2 \right) \label{eq:approxE}	
\end{equation}
and the asymptotic behaviour of the operator norm satisfies 
\begin{equation}
\norm{R_\gamma^F} = \mathcal O\left(\gamma(\delta)^{-\frac{1}{2}}\right).\label{eq:Opnorm}
\end{equation}

\begin{bem}
	The estimates for the approximation error \eqref{eq:approxE}	 and the operator norm \eqref{eq:Opnorm} are not arbitrary; they are met by regularisation methods defined via spectral filtering \cite[Section 3.3]{Rieder_2003}. Examples for such methods are the Tikhonov regularisation or the truncated singular value decomposition (TSVD).
\end{bem}

Using the parameter choice rule $\gamma(\delta) \sim \delta^C$ for some constant $C>0$, we arrive at the (for this choice) optimal parameter $C = \frac{2 \nu}{(\nu +1 )^2}$. Coincidentally, this is also the expected choice as $\delta^\frac{2\nu}{(\nu +1)^2} = \left(\delta^\frac{\nu}{\nu+1}\right)^\frac{2}{\nu+1}$. This parameter choice  implies the error estimate
\begin{align*}
	\norm{x^\dagger - R^F_\gamma R^G_\alpha z^\delta} &\leq \norm{x^\dagger - R_\gamma^Fy^\dagger} + \norm{R_\gamma^F}\norm{y^\dagger -R_\alpha^Gz^\delta}\\
	&= \mathcal O\left( \gamma(\delta)^\frac \nu 2 + \gamma(\delta)^{-\frac 1 2 }\delta^\frac \nu{\nu +1} \right) = \mathcal O\left( \delta^{\left(\frac{\nu}{\nu+1}\right)^2}\right).
\end{align*}
This is, as a regularisation for $A^{\dagger}$ with source condition $x^\dagger \in X_\nu^A$, suboptimal. However, as we split the regularisation into two steps, we expect to amplify the given noise $\delta$ (without regarding constants) to $\delta^\frac \nu {\nu +1}$ when applying the first regularisation. The second regularisation  solves an inverse problem with noise level $\delta^\frac \nu {\nu +1}$ and this is amplified to $\left(\delta^\frac \nu {\nu +1}\right)^\frac \nu {\nu +1} = \delta^{\left(\frac{\nu}{\nu+1}\right)^2}$, the derived estimate.

\subsubsection*{A-Posteriori Parameter Choice Rule}

In the case where the smoothness-parameter $\nu$ is unknown, we propose a modified discrepancy principle as a parameter choice rule. Remind that 
in the classical Morozov discrepancy principle,  a  monotonically decreasing sequence $\alpha_k$ is used  
	and  the parameter  $\alpha = \alpha_k$ is chosen with help of the inequality
	\begin{equation}\label{eqn:ClassMor}
		\norm{AR_{\alpha_k}z^\delta - z^\delta} \leq \tau \delta < \norm{AR_{\alpha_i}z^\delta - z^\delta},
	\end{equation}
	for $i = 1,..., k-1$, where  $\tau >1$ is fixed.  
It is possible to show order optimality for the discrepancy principle for some $\nu$. More concretely, let the regularisation  be given as an (appropriate) filter method $R_\alpha = f_\alpha(A^*A)A^*$ and set $r_\alpha (\lambda) := 1 - \lambda f_\alpha(\lambda)$. We  state the result for the  Morozov principle \cite{Morozov_1968}.
\begin{thm}[\text{\cite[Satz 3.4.1]{Rieder_2003}}]
	Suppose the filter $f_\alpha$ is a regularising filter with qualification $\nu_0$, i.e.
	\begin{equation}
		\lambda ^\frac \nu 2 \abs{r_\alpha(\lambda)} \leq w_\nu(\alpha) \sim \alpha^\frac \nu 2, \text{ for } \alpha \to 0 \label{eqn:ClassMorQual}
	\end{equation}
	for all $0 < \nu \leq \nu_0, \, \nu_0 > 1$ and let
	\begin{align}
		\sup\left\{ \abs{f_\alpha(\lambda)}\mid \lambda \in [0, \norm{A}^2]\right\}\leq \frac C \alpha\label{eqn:ClassMorLim}.
	\end{align}
	Further assume
	\begin{equation}
		\tau > \sup \left\{ \abs{r_\alpha(\lambda)} \mid \alpha >0, \, \lambda \in [0, \norm{A}^2]\right\}.
	\end{equation}
	Moreover, let the sequence of parameters $\alpha_k$ fulfil $\alpha_k = \alpha_{k-1} \theta_{k-1}$
		with $0 < \theta < \theta_k <1$ for all $k \geq 2$. Then the regularisation $(R_\alpha, \alpha)$, with $\alpha$ selected  by the discrepancy principle  \eqref{eqn:ClassMor} is order optimal with regards to $X^A_\nu$ for all $\nu \leq \nu_0 -1$.
	
\end{thm}
We cannot apply the discrepancy principle of Morozov  for the problem with composite operator~\eqref{eqn:OpSplitting}. To see this, we bring to mind that we split the inverse problem into two inverse problems, where we need to select a parameter for each regularisation. 
First, we solve the problem
\[Gy = z\,,\]
where  only the noisy version $z^\delta$ is available with $\norm{z-z^\delta}\leq \delta$. 
Assuming $y = G^\dagger z \in X^G_\nu$ and applying an order optimal regularisation, 
we achieve a reconstruction $y^\varepsilon$ with the error estimate $\norm{y - y^\varepsilon}\leq \varepsilon : = \delta^\frac \nu {\nu +1}$. In this step, we may apply the discrepancy principle. This noise estimate is not available without knowing $\nu$, thus, we  are not able to regularise the second inverse problem
\[Fx = y\]
with help of the Morozov discrepancy principle without knowing $\nu$ in the first place.

In the following, we assume both $F^\dagger$ and $G^\dagger$ to be regularised with filter methods. Therefore, let $f^G_\alpha : [0, \norm{G}^2] \to \R_{\geq 0}$ and $f^F_\gamma : [0, \norm{F}^2] \to \R_{\geq 0}$ be piecewise continuous functions with jump discontinuities. We define\[R^G_\alpha z := f^G_\alpha(G^*G)G^*z = \sum_{n}f^G_\alpha(\mu_n^2)\mu_n \SP{z}{z_n}y_n\]
and
\[R^F_\gamma y := f^F_\gamma(F^*F)F^*y = \sum_{n}f^F_\gamma(\lambda_n^2)\lambda_n \SP{y}{y_n}x_n.\]

Moreover, for the rest of this section, we assume the following on the filters to ensure that $R^G_\alpha$ and $R^F_\gamma$ are regularisations with the desired properties.
\begin{ass} \label{ass_filter}
	\begin{enumerate}
		\item The families $\left\{f^G_\alpha\right\}_\alpha$ and $\left\{f^F_\gamma\right\}_\gamma$ are regularising filters, i.e. 
		\begin{enumerate}
			\item for small regularisation parameter, the regularisations tend to the generalised inverse as
			\begin{align*}
				\lim_{\alpha\to 0}f^G_\alpha(\mu) = \frac 1 \mu, \quad 
				\lim_{\gamma\to 0}f^F_\gamma(\lambda) = \frac 1 \lambda
			\end{align*}
			for all $\mu \in (0, \norm{G}^2]$ and $\lambda \in (0, \norm F^2]$.
			\item the operators $FR^F_\gamma$ and $GR^G_\alpha$ are bounded, i.e. there exist $C_g, C_F > 0$ such that
			\begin{align*}
				\mu \abs{f^G_\alpha(\mu)}\leq C_G, \quad
				\lambda \abs{f^F_\gamma(\lambda)}\leq C_F
			\end{align*}
			for all $\gamma, \alpha >0$, $\mu \in (0, \norm{G}^2]$ and $\lambda \in (0, \norm F^2]$.
		\end{enumerate}
		
		\item Both families of filter functions are, up to a constant, asymptotically limited by $\alpha^{-1}$. This means there exist constants $M_g, M_F>0$  such that 
		\begin{align*}
			\sup_{\mu \in (0, \norm{G}^2]}\abs{f^G_\alpha(\mu)}\leq M_G \alpha^{-1},\,\sup_{\lambda \in (0, \norm{F}^2]}\abs{f^F_\gamma(\lambda)}\leq M_F \gamma^{-1}
		\end{align*}
		for $\gamma, \alpha \to 0$.
		\item The filters $\left\{f^G_\alpha\right\}_\alpha$ and $\left\{f^F_\gamma\right\}_\gamma$ have qualification $\nu_0$, i.e. for all $\nu \leq \nu_0$, $\nu_0 > 1$ 
		\begin{align*}
			\sup_{\mu \in (0, \norm{G}^2]} \mu^\frac{\nu}{2}\abs{p^G_\alpha(\mu)} \leq D_G(\nu)\alpha^\frac\nu 2, \; \alpha \to 0,\\
			\sup_{\lambda \in (0, \norm{F}^2]} \lambda^\frac{\nu}{2}\abs{p^F_\gamma(\lambda)} \leq D_F(\nu)\gamma^\frac\nu 2, \; \gamma\to 0.
		\end{align*}
In particular, we assumed the qualification of $R^F_\gamma$ and $R^G_\alpha$ to be the same to simplify the following discussion. However, the principle can be studied and applied with regularisations with different qualifications. Here, we denote $p_\alpha^G(\mu) = 1 - \mu f^G_\alpha(\mu)$ and $p_\gamma^F(\lambda) = 1 - \lambda f^F_\gamma(\lambda).$
		
	\end{enumerate}
	
\end{ass}

With these assumptions, we achieve the following result.
\begin{thm}\label{thm:RegMorozov}
	Let the Assumptions \ref{ass_filter} hold and  $\tau^G > \sup_{\alpha >0, \, \mu \in [0, \norm{G}^2]}\left\{ p_\alpha^G(\mu)\right\} \geq 1$. Moreover,  we  assume
	\[\tau^F > \sup_{\gamma >0, \, \lambda \in [0, \norm{F}^2]}\left\{ p_\gamma^F(\lambda)\right\}C_G =: S_F C_G.\]
	Further, set $\alpha_{k+1} = \vartheta ^G _k \alpha_k$ and $\gamma_{n+1} = \vartheta ^F _n \gamma_n$, with $0 < \theta \leq  \vartheta_k ^G, \vartheta_n ^F <1$. 
	Choose the regularisation parameter $\alpha = \alpha_{k^*}$  for $R^G_\alpha$ such that
	\begin{equation}\label{eqn:ownMor1}
		\norm{GR^G_{\alpha_{k^*}}z^\delta - z^\delta}\leq \tau^G\delta < \norm{GR^G_{\alpha_{i}}z^\delta - z^\delta}, \; i = 1, ..., k^*-1
	\end{equation}
	and the parameter $\gamma = \gamma_ {n^*}$ for $R^F_\gamma$ such that
	\begin{equation}\label{eqn:ownMor2}\norm{GFR^F_{\gamma_{n^*}}R^G_{\alpha_{k^*}}z^\delta - GR^G_{\alpha_{k^*}}z^\delta} \leq \tau^F \delta < \norm{GFR^F_{\gamma_{i}}R^G_{\alpha_{k^*}}z^\delta - GR^G_{\alpha_{k^*}}z^\delta}, 
	\end{equation}
	$i = 1, ..., n^*-1$. Then the resulting regularisation $R^F_\gamma R^G_\alpha$ fulfils
	\[\norm{x^\dagger - R^F_\gamma R^G_\alpha z^\delta} = \mathcal O\left( \delta^\frac{\nu-1}{\nu+1} \rho^\frac{1}{(\nu +1)^2}\right)\]
	for all $z = \abs{A}^\nu w \in X^A_\nu$, $\norm{z - z^\delta}\leq \delta$, $\norm{w} \leq \rho$ and $\nu \in (1, \nu_0-1]$.
\end{thm}
We call the parameter choice rule \eqref{eqn:ownMor1} - \eqref{eqn:ownMor2}  \textbf{\nameMethode}. We will prove the rate of convergence with help of some lemmata.

\begin{lem}\label{lem:calcDataErr}
	Define $q(z) := \left(GFR^F_\gamma R^G_\alpha - GR^G_\alpha\right)z$. Then we have 
	
	\begin{equation}\label{eqn:lem:DataErr1}
		q(z) = -\sum_n  p^F_\gamma(\lambda_n^2)\left(f^G_\alpha(\mu_n^2)\mu_n^2\right)\SP{z}{z_n}z_n,
	\end{equation}
	$\norm{q(z) }\leq \sup_{\gamma, \lambda} p^F_\gamma(\lambda) C_G \norm{z}$,
	\begin{equation}\label{eqn:lem:DataErr2}
		q(A \abs{A}^\nu w) = -\sum_n  \lambda_n^{\nu +1}\mu_n^{\nu +1}p^F_\gamma(\lambda_n^2)\left(f^G_\alpha(\mu_n^2)\mu_n^2\right)\SP{w}{x_n}z_n,
	\end{equation}
	and $\norm{q(A \abs{A}^\nu w)} \leq D_F(\nu +1) s^\frac{\nu +1}{2}C_G \norm G^{\nu +1}\norm w$ for $s \to 0$.
\end{lem}
\begin{proof}
	We can rewrite the operator $q$ with help of the singular value decompositions \eqref{eqn:SVD_FG} of $F$ and $G$. We expand both parts of $q$ separately and obtain
	\[GR^G_\alpha(z) = Gf^G_\alpha(G^*G)G^*z = \sum_nf^G_\alpha(\mu^2_n)\mu_n^2\SP{z}{z_n}z_n\]
	and
	\begin{align*}
		GFR^F_\gamma R^G_\alpha z &= GFf^F_\gamma (F^*F)F^*f^G_\alpha (G^*G)G^*z \\
		&=\sum_n f^F_\gamma (\lambda_n^2)\lambda_n^2 f^G_\alpha (\mu_n^2)\mu_n^2 \SP{z}{z_n}z_n.
	\end{align*}
	Both results combined give
	\begin{align*}
		q(z) &= \sum_n \left( f^F_\gamma (\lambda_n^2)\lambda_n^2 f^G_\alpha(\mu_n^2)\mu_n^2 - f^G_\alpha(\mu^2_n)\mu_n^2\right)\SP{z}{z_n}z_n\\
		&= \sum_n \left( f^F_\gamma (\lambda_n^2)\lambda_n^2 - 1\right) f^G_\alpha (\mu_n^2)\mu_n^2 \SP{z}{z_n}z_n\,,
	\end{align*}
	which proves \eqref{eqn:lem:DataErr1}. This representation results in an upper bound for the squared norm
	\[
		\norm{q(z)}^2 = \sum_n \underbrace{\left( f^F_\gamma(\lambda_n^2)\lambda_n^2 - 1\right)^2}_{\leq\left(\sup_{\gamma, \lambda} p^F_\gamma(\lambda)\right)^2} \underbrace{\left(f^G_\alpha(\mu_n^2)\mu_n^2\right)^2}_{\leq C_G^2} \abs{\SP{z}{z_n}}^2\leq \left(\sup_{\gamma, \lambda} p^F_\gamma(\lambda) C_G \norm z \right)^2.
	\]
	The second representation \eqref{eqn:lem:DataErr2} follows directly from
	\[A\abs{A}^\nu w = \sum_k \lambda_n^{\nu +1} \mu_n^{\nu +1} \SP{w}{x_n}z_n\,.\]
	We conclude that
	\begin{align*}
		\norm{q(A\abs{A}^\nu w)}^2 &= \sum_n \underbrace{\left(\lambda_n^{\nu +1} p_\gamma^F(\lambda_n^2)\right)^2}_{\leq\left(D_F(\nu+1)\gamma^\frac{\nu +1}{2}\right)^2} \underbrace{\left(f^G_\alpha(\mu_n^2)\mu_n^2\right)^2}_{\leq C_G^2} (\underbrace{\mu_n^{\nu+1}}_{\leq\norm{G}^{\nu+1}})^2 \abs {\SP{w}{x_n}}^2\\
		&\leq \left(D_F(\nu+1)\gamma^\frac{\nu +1}{2} C_G \norm{G}^{\nu+1} \norm{w}\right)^2
	\end{align*}
	for $\gamma \to 0$.
\end{proof}

\begin{lem}\label{lem:calcApproxErr}
	Define $p(x) = (\mathrm{id} - R^F_\gamma R^G_\alpha GF)x$. Then it holds
	\[p\left(\abs{A}^\nu w\right) = \abs{A^\nu}p(w).\]
\end{lem}
\begin{proof}
	Using
	\[p(x) = \sum_n \left( 1 - f^F_\gamma (\lambda_n^2)\lambda_n^2 f^G_\alpha(\mu_n^2)\mu_n^2 \right) \SP{x}{x_n}x_n\]
	and inserting $x = \abs{A}^\nu w = \sum_n \lambda_n^\nu \mu_n^\nu \SP{w}{x_n}x_n$, we arrive at
	\[
		p(\abs{A}^\nu w ) = \sum_n \lambda_n^\nu \mu_n^\nu\left( 1 - f^F_\gamma(\lambda_n^2)\lambda_n^2 f^G_\alpha(\mu_n^2)\mu_n^2 \right) \SP{w}{x_n}x_n = \abs{A}^\nu p(w).
	\]	

\end{proof}
For the proof of Theorem \ref{thm:RegMorozov} we need the interpolation equality and the Picard condition, see Appendix~\ref{subsec:RegThBasics}.

\begin{proof}[Proof of Theorem~\ref{thm:RegMorozov}]
	Let $z \in \mathcal R(A)$ and denote $x^\dagger = A^\dagger z$. We assume $x^\dagger = \abs{A}^\nu w \in X_\nu^A$ for some $\nu \leq \nu_0-1$ and $\norm{w} \leq \rho$. Let $\alpha_{k^*} =:\alpha$ 
	 and $\gamma_{n^*} =: \gamma$ be the parameter chosen 
	 by the double discrepancy principle \eqref{eqn:ownMor1} and \eqref{eqn:ownMor2}. We define $y^\dagger = G^\dagger z$ and have therefore $x^\dagger = F^\dagger y^\dagger$.
		The assumptions of this section ensure that $R^G_\alpha$ is an order-optimal regularisation, as the parameter choice rule \eqref{eqn:ownMor1} coincides with the Morozov discrepancy principle \eqref{eqn:ClassMor}. Thus, using $\lesssim$ as an abbreviation for "less or equal up to a constant not dependent on $\delta$ and $\rho$", we have $\norm{y^\dagger - R^G_\alpha z^\delta}\lesssim \delta^\frac{\nu}{\nu+1}\rho ^\frac 1 {\nu +1}$.

	Next, we apply the triangle inequality
	\begin{equation}\label{eqn:TriangProofMor}
		\norm{x^\dagger - R^F_\gamma y^\dagger}\leq \norm{x^\dagger - R^F_\gamma R^G_\alpha z} + \norm{R^F_\gamma}\norm{R^G_\alpha z - y^\dagger}
	\end{equation} and analyse both parts separately. We start with the first term. Lemma~\ref{lem:calcApproxErr} and the interpolation inequality \eqref{eqn:IntEq} imply
	\begin{align}\label{eqn:xdag-xdel}
		\norm{x^\dagger - R^F_\gamma R^G_\alpha z} &= \norm{x^\dagger - R^F_\gamma R^G_\alpha GFx^\dagger}=\norm{\abs{A}^\nu (1-R^F_\gamma R^G_\alpha GF)w}\notag\\
		&\leq \norm{\abs{A}^{\nu+1} (1-R^F_\gamma R^G_\alpha GF)w}^\frac{\nu}{\nu+1}\norm{(1-R^F_\gamma R^G_\alpha GF)w}^\frac{1}{\nu+1}
	\end{align}
	and hence
	\begin{align}\label{eqn:xdag-xdel2}
		&\hspace{-1em}\norm{\abs{A}^{\nu+1} (1-R^F_\gamma R^G_\alpha GF)w} = \norm{A (1-R^F_\gamma R^G_\alpha GF)x^\dagger} = \norm{z - AR^F_\gamma R^G_\alpha z}\notag\\
		&\leq\norm{ AR^F_\gamma R^G_\alpha z^\delta - GR^G_\alpha z^\delta} + \norm{GR^G_\alpha z^\delta - z^\delta} + \norm{AR^F_\gamma R^G_\alpha (z-z^\delta) - (z-z^\delta)}\notag\\
		&\leq \tau^F\delta  + \tau^G\delta+ \norm{(1- GFR^F_\gamma R^G_\alpha)(z-z^\delta)}\notag\\
		&\leq (\tau^F + \tau^G + 1 + C_G C_F)\delta.
	\end{align}
	Here, we used the defining inequalities of the double discrepancy principle \eqref{eqn:ownMor1} and \eqref{eqn:ownMor2}.
	Plugging \eqref{eqn:xdag-xdel2} in \eqref{eqn:xdag-xdel}, we have bound the first term of \eqref{eqn:TriangProofMor} by $\mathcal O(\delta^\frac{\nu}{\nu+1}\rho^\frac{1}{\nu})$.

	We can limit the second term of \eqref{eqn:TriangProofMor} by $\norm{R^F_\gamma}\norm{y^\dagger - R^G_\alpha z}\lesssim \norm{R^F_\gamma} \rho ^\frac 1 {\nu +1} \delta^\frac\nu {\nu +1}.$
	We now investigate the operator norm of $R^F_\gamma$ in two cases. 
	
	In the first case, we  assume that $n^* \geq 2$. Writing $\sigma := \gamma_{n^* -1} = \frac \gamma {\vartheta^F_{n^*-1}}$, the right inequality of the discrepancy principle \eqref{eqn:ownMor2} together with Lemma~\ref{lem:calcDataErr} yields
	
	\begin{align*}
		\tau^F\delta &< \norm{GFR^F_{\sigma}R^G_\alpha z^\delta - GR^G_\alpha z^\delta}\\
		&\leq\norm{\left(GFR^F_{\sigma}R^G_\alpha - GR^G_\alpha\right)z} + \norm{\left(GFR^F_{\sigma}R^G_\alpha - GR^G_\alpha\right)(z^\delta -z)}\\
		&\leq D_F(\nu +1) \sigma^\frac{\nu +1}{2}C_G \norm G \rho + \sup_{\gamma, \lambda} p^F_\gamma(\lambda) C_G \delta\\
		&= D_F(\nu+1) \sigma^\frac{\nu +1}{2}C_G \norm G \rho + S_F C_G \delta\,.
	\end{align*}
	This leads to 
	\[\left(\tau^F - S_F C_G\right)\delta \leq D_F(\nu+1) \theta ^{-\frac{\nu +1}{2}}s^\frac{\nu +1}{2}C_G \norm G \rho\]
	and
	\[s^{-\frac 1 2} \leq \left(\frac{D_F(\nu+1) C_G \norm{G}}{\tau^F - S_F C_G}\right)^\frac{1}{\nu+1}\theta^{-\frac 1 2} \rho^\frac 1 {\nu +1} \delta ^{-  \frac 1 {\nu +1}}.\]
	
	In the case of $\gamma = \gamma_1$ we assume that $z \neq 0$, and for sufficiently small $\delta$ it holds
	\[\delta \leq \norm{z} = \norm{\abs{A}^\nu w} \leq \norm A^\nu \rho,\]
	and therefore
	\[\gamma^{-\frac 1 2} = \gamma_1^{-\frac 1 2} \leq \gamma_1^{-\frac 1 2} \norm{A} \rho ^\frac 1 {\nu +1} \delta^{-\frac 1 {\nu +1}}.\]
	
	In both cases Lemma~\ref{lem:OpNormRF} implies
	\begin{equation}\label{eqn:NormRF}
		\norm{R^F_\gamma} \leq \sqrt{C_F M_F}\gamma^{-\frac 1 2} \lesssim\rho ^\frac 1 {\nu +1} \delta^{-\frac 1 {\nu +1}}.
	\end{equation}
	
	We conclude with \eqref{eqn:xdag-xdel}, \eqref{eqn:xdag-xdel2}, \eqref{eqn:NormRF} and the bound of the order-optimal regularisation $R^G$ that
	\begin{align*}
		\norm{x^\dagger - R^F_\gamma R^G_\alpha z^\delta} &\leq \norm{x^\dagger - R^F_\gamma R^G_\alpha z} + \norm{R^F_\gamma}\left(\norm{R^G_\alpha z - y^\dagger}  + \norm{y^\dagger - R^G_\alpha z^\delta}\right)\\
		&\lesssim \delta^\frac{\nu}{\nu +1} \rho^\frac 1 {\nu +1} + \delta^{-\frac{1}{\nu +1}}\rho^\frac 1 {\nu +1}\left(\delta^\frac{\nu}{\nu+1}\rho^\frac{1}{\nu +1}\right)\\
		&= \mathcal O \left( \delta^\frac{\nu-1}{\nu +1}\rho^\frac{1}{(\nu +1)^2}\right).
	\end{align*}
\end{proof}

Up to now, it is unclear if the parameter choice rules \eqref{eqn:ownMor1} and \eqref{eqn:ownMor2} always lead to admissible parameters. We have $FR^F_\gamma y - y \to \mathrm P_{\mathcal R(F)^\perp}x$ and similar $GR^G_\alpha z - z\to \mathrm P_{\mathcal R(G)^\perp}z$ for $\gamma,\,\alpha \to 0$. Thus, for $z \in \mathcal R(A)^\perp$ this shows $\norm{GR^G_\alpha z^\delta - z^\delta} \to \norm{\mathrm P_{\mathcal R(A)^\perp}(z-z^\delta)} \leq \delta$ and the choice of the Morozov discrepancy principle and  also the rule \eqref{eqn:ownMor1} are always possible. Regarding \eqref{eqn:ownMor2}, we obtain
\begin{align*}
	GFR^F_{\gamma}R^G_{\alpha_{k^*}}z^\delta - GR^G_{\alpha_{k^*}}z^\delta &= G \left( \left(FR^F_{\gamma} - \mathrm{id}\right) \left(\sum_n \SP{y_n}{R^G_{\alpha_{k^*}}z^\delta}y_n\right)\right)\\
	&\to G  \left(\sum_n \SP{y_n}{R^G_{\alpha_{k^*}}z^\delta}\underbrace{\mathrm{P}_{\mathcal R(F)^\perp}(y_n)}_{= 0}\right)= 0
\end{align*}
for $\gamma \to 0$. This shows that the parameter choice rule \eqref{eqn:ownMor2} is well-defined, too. Moreover, even the choice of $\tau^F <1$ results in a finite iteration.
Similar to the Morozov discrepancy principle, we have a suboptimal result for finite qualification.

\section{Numerical Experiments}\label{sec:NumExp}
In this section, we evaluate the proposed \nameRegmethode ~method on two examples. We test it once on the pre-clinical imaging technique of magnetic particle imaging in Subsection~\ref{sec:NumerikMPI}, which we will briefly introduce in Subsection~\ref{sec:MPI}. Since no smoothness information is available for the MPI example,  we apply the \nameMethode ~\eqref{eqn:ownMor1} and \eqref{eqn:ownMor2} on a convolution example in Subsection~\ref{sec:NumerikConv}.

\subsection{Basics for Model-based Reconstruction in MPI}\label{sec:MPI}

We state a simple model of the MPI forward operator. As already mentioned in the introduction, the aim is the reconstruction of  the spatial distribution of a magnetisable tracer inside the patients body. Applying an external magnetic field, the change of the particle magnetisation induces a voltage which is measured with help of receive coils. More explicitly, the Biot-Savard law implicates that the particle concentration is linked to the measured voltage of the receive coils via
\[u(t) = - \mu_0 \int_\Omega c(\mathbf r) \partial_t \bar {\mathbf m}(\mathbf r, t) \cdot \mathbf p \, \mathrm d \mathbf r\,,\]
see  \cite{KnoppBuzug_2012} for details. 
Here, $\mu_0$ denotes the magnetic permeability on free space, $\Omega$ the area enclosed by the respective receive coil, $\bar{\mathbf m}$ the mean magnetic moment and $\mathbf p$ the coil sensitivity. The mean magnetic moment can be modelled in multiple ways, for example as the solution of the Fokker-Plank equation \cite{Kluth_etal_2019} or the Landau-Lifshitz-Gilbert equation \cite{Kaltenbacher_etal_2019}. In the following we will use the much simpler Langevin model of paramagnetism, where 
the mean magnetic moment is described using the Langevin function 
\begin{align*}
\mathcal L: \R \to \R\,,\quad
\lambda \mapsto \begin{cases}
\coth(\lambda) - \frac 1 \lambda & \lambda \neq 0\\
0 & \text{otherwise.}
\end{cases}
\end{align*}
However, this model ignores relaxation effects of the particles.
The derivative of the Langevin function resembles a bell-curve, i.e. the change of magnetisation is biggest where the magnetic field vanishes. 
	Thus, the particles in that area contribute most to the signal and particles in magnetic saturation the least, see also Figure~\ref{fig:LangevinDer}. 

\begin{figure}
	\centering
	\includegraphics[]{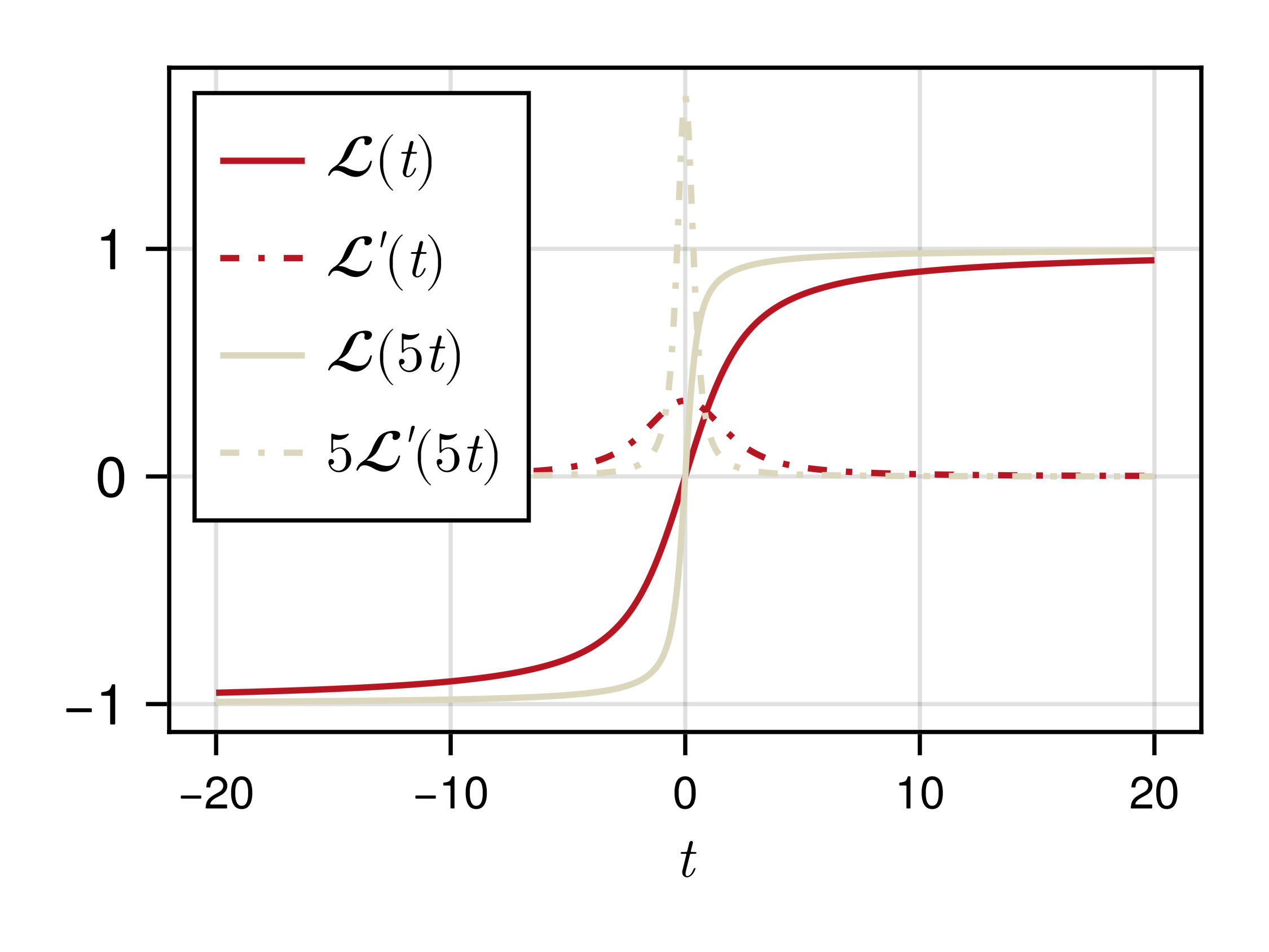}
	\caption{The Langevin function $\mathcal L$ links the modulus of the mean particle magnetisation given an external magnetic field. Its derivative $\mathcal L'$ is the principal component of the convolution kernel in the MPI application.}
	\label{fig:LangevinDer}
\end{figure}

The scanner and its setting is encoded in the magnetisation. Thus, a change in the scanning environment results in a different forward operator. In the following, we will assume that the area, where the magnetisation vanishes, resembles a line which is rotated and translated during the measuring process. This set-up is called field-free line (FFL) MPI instead of a point-shaped volume called field-free point (FFP). 

\begin{figure}[h]
	\centering
	\begin{subfigure}{0.45\linewidth}
	\centering
	\includegraphics[scale = 0.1875]{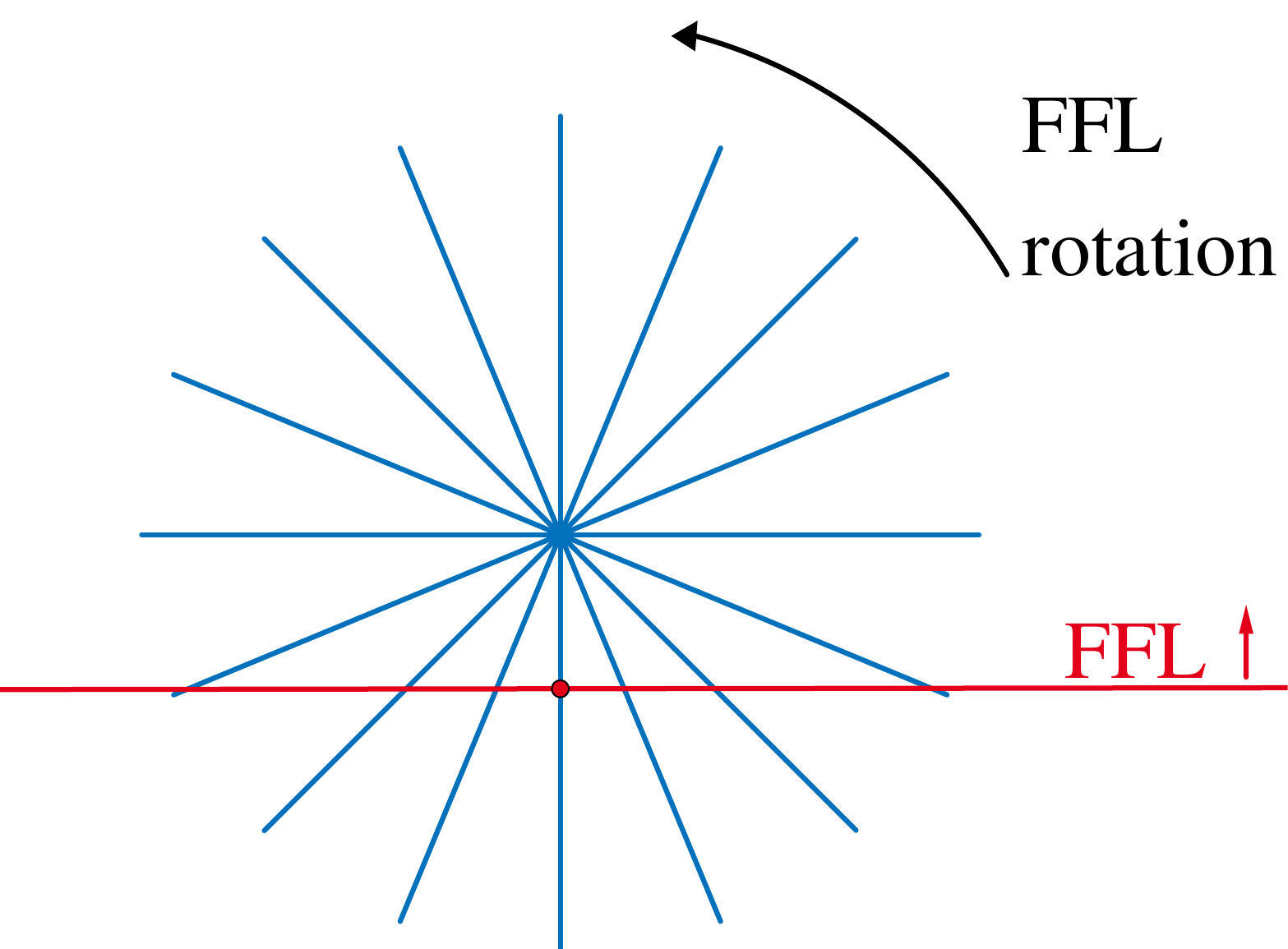}
	\subcaption{Sequential line rotation}
\end{subfigure}\begin{subfigure}{0.45\linewidth}
	\centering
	\includegraphics[scale = 0.1875]{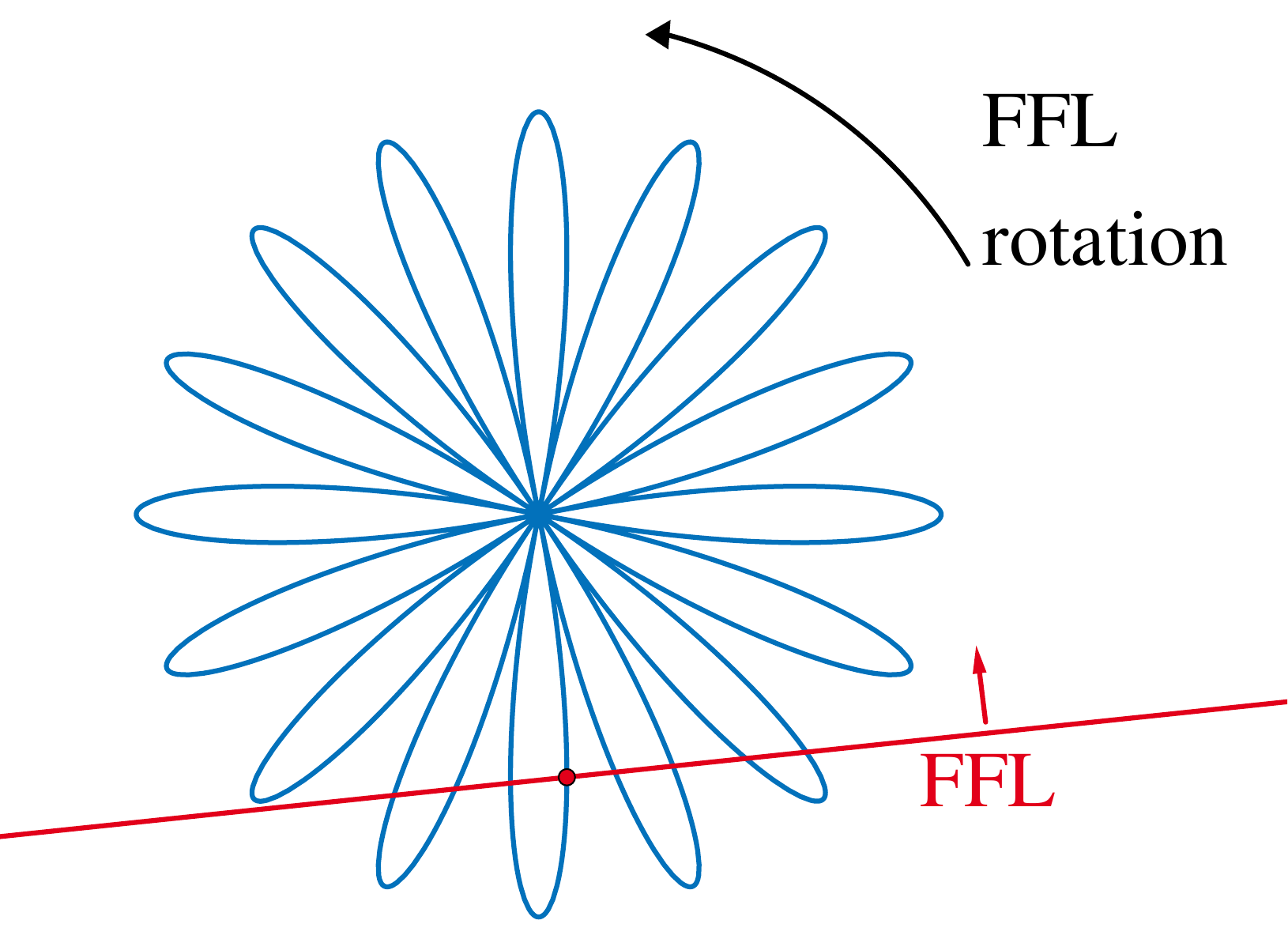}
	\subcaption{Simultaneous line rotation}
\end{subfigure}
	\caption{\cite{Blanke_2022}: Movement of the field-free line through 2D space for different set-ups. (a) The angle of the line is piecewise constant, (b) angle changes uniformly during the measuring process.}
	\label{fig:FFL}
\end{figure}

In \cite{Knopp_2011}, FFL MPI with sequential line rotation was connected to the Radon transform. This was later expanded to simultaneous line rotation \cite{BlankeBrandt_2023}. See also Figure~\ref{fig:FFL} for an illustration. In preparation to restate these results, we define the Radon transform and related operators. In the following, we use the Schwartz space
\[\mathscr S(\R^2) := \set{ f : \R^2 \to \R }{ \lim_{\abs{x} \to \infty} x^\alpha D^\beta f (x) = 0 \; \forall \alpha, \beta \in \N_0^2}.\]
We adjust the definition to different spaces accordingly. We use the two-dimensional Fourier transform $\mathcal F_2$ on functions $f:\R^2 \to \R$ and the one-dimensional Fourier transform $\mathcal F_1$ on functions $f:\R\times [0,\pi)$ on the first dimension. Further, we need the Radon transform,
\begin{align*}
		\mathcal R: \mathscr S(\R^2) &\to \mathscr S(\R\times [0,\pi))\\
		\mathcal R f (s, \varphi) &= \int_\R f(s \cos \varphi + t \sin \varphi, s \sin \varphi + t \cos \varphi) \,\mathrm d t,
	\end{align*}
	the backprojection,
	\begin{equation}\label{eqn:DefBP}\begin{split}
		\mathcal B: \mathscr S ( \R \times [0, \pi)) &\to \mathscr S(\R^2)\\
		\mathcal Bg(x,y) &= \frac 1 \pi \int_0^\pi g(x\cos \varphi + y \sin \varphi, \varphi)\,\mathrm d \varphi,
	\end{split}\end{equation}
	and the weighted Radon transform
	\begin{align*}
		\tilde{\mathcal R}: \mathscr S(\R^2) &\to \mathscr S(\R\times [0,\pi))\\ \tilde {\mathcal R}f(s, \varphi) &= \int_\R f(s \cos \varphi - t \sin \varphi, s \sin \varphi + t \cos \varphi) t \,\mathrm d t.
	\end{align*}

Using the well-known Fourier-Slice theorem, see Theorem \ref{thm:FourierSlice} in the appendix, we can obtain reconstruction formulas.
In comparison to the classical formulation, see for example  Theorem II.2.1 in \cite[]{Natterer1986}, we add a function $k \in \mathscr S(\R)$ in the following theorem, which is used  later to modify the formula of the filtered backprojection to the MPI setting.
\begin{thm}\label{thm:MPI_FBPConv}
	For $\alpha < 2$ and $f \in \mathscr S(\R^2)$ and $k \in \mathscr S(\R)$, we have 
	\begin{align*}
		k\left(\norm{\xi}\right)\norm{\xi}^{-\alpha} \mathcal F_2 f(\xi) = \frac 1 2 \mathcal F_2( \mathcal B(S))(\xi) ,
	\end{align*}
	where $\mathcal F_1 S (\xi, \varphi)= k(\abs{\xi})\abs{\xi}^{1-\alpha}\mathcal F_1 \mathcal Rf(\xi, \varphi)$.
\end{thm}
\begin{proof}
	With help of the Fourier Slice theorem, we see directly that 
	\begin{align*}
	\frac 1 2 \mathcal B (S)(x, y)	& =\frac 1 {2\pi} \int_0^\pi S(x \cos \varphi + y \sin\varphi)\, \mathrm d \varphi  \\
		&= \frac 1 {(2\pi)^2} \int_0^\pi \int_\R e^{i(x\cos\varphi + y \sin \varphi)r} k(\abs r)\abs{r} \abs{r}^{-\alpha}\underbrace{\mathcal F_1\mathcal Rf(r, \varphi)}_{=\mathcal F_2 f(r\cos \varphi, r \sin \varphi)} \mathrm d r \mathrm d \varphi \\
		&= \frac 1 {(2\pi)^2} \int_{\R^2} e^{i(x\xi_1 + y \xi_2)}k(\norm{\xi})\norm{\xi}^{-\alpha}\mathcal F_2 f(\xi)\,\mathrm d(x,y)\\
		& = \mathcal F_2^{-1} \left( \xi \mapsto k\left(\norm{\xi}\right)\norm{\xi}^{-\alpha} \mathcal F_2 f(\xi)\right),
	\end{align*}
	showing the claim.
\end{proof}
This also holds for $L^p$ spaces of appropriate order. In particular in the $L^p$ setting, fixing $\alpha = 0$ and $k \equiv 1$, we receive 
\begin{equation}
	f(\xi) = \frac 1 2 \mathcal B \left(\mathcal F_1^{-1}\left[ \abs s \mathcal F_1\mathcal Rf(s, \varphi) \right]\right)(\xi).
\end{equation} This formula is numerically unstable as 
the spectrum of the measurement $\mathcal Rf$ is filtered with the absolute value function.
Thus, the noise is amplified arbitrarily much
such that  a low-pass filter is needed in order to avoid the amplification of  high-frequency noise components. Using a low-pass filter $F_L \in \mathscr S(\R)$ with  $\mathrm{supp}(F_L) \subset [-L, L]$ instead of the absolute value, we arrive at
\begin{align}
	f(\xi) &= \frac 1 2 \mathcal B \left(\mathcal F_1^{-1}\left[ F_L(s) \mathcal F_1\mathcal Rf(s, \varphi) \right]\right)(\xi) = \frac 1 2 \mathcal B \left(\mathcal F_1^{-1} F_L(s) *_1 \mathcal Rf(s, \varphi) \right)(\xi),
\end{align}
which, discretised appropriately, is the 
formula of the filtered backprojection. Note that typically $F_L \in L^\infty(\R)$. We can similarly use $k$ in the above theorem to stabilise the reconstruction. For arbitrary $k\in \mathscr S(\R)$, we obtain for $\alpha = 0$ the  formula
\begin{equation}\label{eqn:FBPandKonv}
	\mathcal F_2^{-1} k(\norm \xi) *_2 f(\xi) = \frac 1 2 \mathcal B \left(\mathcal F_1^{-1} F_L(s) *_1 \mathcal F^{-1}k(\abs s) *_1 \mathcal Rf(s, \varphi) \right)(\xi).
\end{equation}
Now we can cite the desired result for the MPI operator \cite{BlankeBrandt_2023}. Under idealised settings, the signal $u(t)$ can be represented via
	\begin{equation}
		u(t) = \left[K_1 \circ \mathcal R + K_2 \circ \tilde {\mathcal R} + K_3 \circ \mathcal R\right] c(t), \label{eq:MPI_forwardmodel}
	\end{equation}
	where $K_i$ are convolution operators. This means that $K_i: \mathscr S(\R \times [0, \pi)) \to \mathscr S(\R)$ is given with help of appropriate functions $k_i \in \mathscr S(\R)$ as
	\[K_i \circ f (t) = \left[k_i( \cdot) * f(\cdot, \varphi(t))\right] (s(t)).\]
The function $s(t)$ denotes the displacement of the FFL from the origin and $\varphi(t)$ its angle to the $y$-axis.
See again Figure~\ref{fig:FFL} for a sketch of the movement of the FFL. 
The preceding result \cite[Theorem 1]{Knopp_2011} assumes sequential line rotation. In this case and in the above notation we have 
\[u(t) = K_1\circ \mathcal R (c(t)).\]

Reading $k$ as a convolution kernel, Theorem~\ref{thm:MPI_FBPConv} allows us to change the order of operations reconstructing MPI data. Equation~\eqref{eqn:FBPandKonv} tells us that the filtered backprojection of MPI data, i.e. the right-hand side, is equal to the convoluted reconstruction. This means we can first invert the Radon transform and afterwards deconvolve which is counter-intuitive to the order of regularisation coming from the order of operations given in~\eqref{eq:MPI_forwardmodel}. However, this allows us to deconvolve in image space instead of the sinogram space which is computationally beneficial.

\subsection{Experiment 1: Model-based MPI Reconstruction}\label{sec:NumerikMPI}

We want to test the reconstruction scheme stated above on simulated MPI data. 

	Neglecting two summands in \eqref{eq:MPI_forwardmodel}, the simplified forward  model is given by  the composed operator 
	\[A = K \circ \mathcal R\,.\]
	 In \cite{Blanke_2022} it was shown that the different addends are of different magnitude under reasonable assumptions  which allows for this simplification. The convolution operator has the form 
	\begin{align*}
		Kf(t) &= s(t) \left[\bar{\mathbf m}(G \cdot) * f(\cdot, \varphi(t))\right](s(t)) 
	\end{align*}
	with convolution kernel  $\bar{\mathbf m} = C \mathcal L'( C \cdot)$ for an appropriate constant $C$. 
	A further regularisation step is necessary to treat the factor $s(t)$. 
We apply Tikhonov regularisation to address this factor. A visualisation of the kernel $\mathbf{\bar m}$ is depicted in Figure~\ref{fig:LangevinDer} for different constants $C$.
	
	It is computationally costly to apply the regularisation as suggested above, i.e. first deconvolve and afterwards use a regularisation for the Radon transform, as the measurement is sampled along a curve in sinogram space which does not correspond to a regular grid. In terms of Theorem~\ref{thm:MPI_FBPConv} we may change this order of operation, when we set the function $k$ as the inverse Fourier transform of the convolution kernel of operator $K$ leading to
	\begin{equation}\label{eqn:MPI_kernelformula}
		Af(t) = \left[ \mathcal F^{-1}_1 k * \mathcal R f( \cdot, \varphi(t) \right](s(t)). 
	\end{equation}
One possibility is to use the backproject-filter method (BPF) as a regularisation for the Radon transform. Here, we first backproject the measurement and afterwards filter it. The occurring regularisation parameter is the inverse of the band width of the filter. This saves on the computational cost in the same way, as the result of the backprojection is given on a regular grid and both the filtering step as well as the deconvolution step afterwards are computationally cheap.
We sketch the BPF algorithm in Algorithm~\ref{alg:BFP}. We want to remind the reader of both the formula of the backprojection~\eqref{eqn:DefBP}
and the way the signal $u(t)$ is given on a curve in sinogram space
\[u(t) = h\left( s(t), \varphi(t)\right),\]
where $h:\R \times [0, \pi) \to \R$ is an appropriate function and  $\Gamma(t) = (s(t), \varphi(t))$ is the excitation function of the MPI scanner composed of the distance $s$ of the FFL to the origin and the angle of the FFL $\varphi$ to the $y$-axis.

\begin{algorithm}
	\caption{Backproject-Filter Method}\label{alg:BFP}
	\begin{algorithmic}[1]
		\Require MPI excitation function $\Gamma(t) = (s(t), \varphi(t))$, Grid $(x_i, y_i)_i$ for evaluation, low-pass filter $F_L$
		\For{i = 1,...}
		\State Compute $\tilde T_i$ such that for all $t\in \tilde T_i$ we have
		\[s(t) = x_i \cos \varphi(t) + y_i \sin \varphi (t).\]
		\State Approximate the backprojection with data $\left\{h(\Gamma(t))\right\}_{t \in \tilde T_i}$, for example via trapezoid rule
		\[\mathcal B h(x_i,y_i) \approx \frac 1 {\pi \abs{\tilde T_i}}\sum_{t\in \tilde T_i}h(\Gamma(t)),\]
		interpolating when necessary.
		\EndFor
		\State Filter the backprojected data
		\[\left\{f(x_i,y_i)\right\}_i \approx \left\{\mathcal F^{-1}F_L(\norm{(x_i,y_i)})\right\}_i * \left\{\mathcal B h(x_i,y_i)\right\}_i\]
		\Return $\left\{f(x_i,y_i)\right\}_i$
	\end{algorithmic}
\end{algorithm}

On the other hand, it is possible to apply a TSVD using the singular value decomposition of the Radon transform, see Theorem~\ref{thm:SVDRadon}. 
We approximate the inner product of the measurement with the corresponding eigenfunctions by the discrete sum of point-evaluations on the grid-points,
\begin{align*}
&\int_{[-10, 10] \times \mathbb S^1} w^{-1}(s, \varphi) f(s, \varphi) \overline{g(s, \varphi)}\uprd(s, \varphi) \\
&\quad\approx \sum_t w^{-1}(s(t), \varphi(t))\abs{s'(t)} f(s(t), \varphi(t)) \overline{g(s(t), \varphi(t))}.
\end{align*}
The correction $\abs{s'(t)}$ is needed to compensate that the samples lie not on a uniform grid.

For testing purposes, we use a spiral shaped phantom simulated in the MFS5 software \cite{SoftwareVogel}. More precisely, this phantom is spread over 500 points along a spiral of minimal radius 1 with 2 loops and a maximal radius of 7. To this simulated data, we added Gaussian noise with zero mean and standard deviation of $10\%$ of the amplitude of the signal.

As a benchmark method, we use a reconstruction from a system matrix approach. Here, we regularise the operator $A^\dagger$. The reconstruction is computed with help of the regularised normal equation with regularisation parameter $\lambda >0$,
\[(A^*A + \lambda I)x = A^*z^\delta,\] solved via the Kaczmarz algorithm \cite{Kaczmarz_1937} implemented in Julia \cite{RegularizedLeastSquares.jl, julia}. In MPI reconstruction, the Kaczmarz algorithm is a standard procedure, as typically, system matrices are 
close to orthogonal and the algorithm converges quickly. Moreover, using early stopping as an additional regularisation is possible, and thus the procedure is computationally very cheap. To save on memory and reduce the computational load, the columns of the matrix are Fourier-transformed and we only use the $330$ frequencies which contribute most to the signal.

In total, we conduct 5 tests for the MPI example, see Figure \ref{fig:recosMPI}. First, we do the benchmark reconstruction (a) with the Kaczmarz algorithm iterating 200 times to ensure convergence. 
Applying the Fourier transformation and the frequency selection  amplified the noise level. 
The regularisation parameter was chosen with the discrepancy principle of Morozov using the hyperparameters of Table~\ref{table:Hyper_parameters_MPI} and the exact noise level of the Fourier-transformed data. The selected parameter was $\lambda = 1.3411\mathrm{e}{-4}$. 

{\renewcommand{\arraystretch}{1.2}
\begin{table}
\centering
\begin{tabular}{|l|c|c|c||c|c|c|c|c|c|}
	\hline 
	& \multicolumn{3}{l||}{Morozov principle} & \multicolumn{6}{c|}{Double discrepancy principle}\tabularnewline
	\hline 
	Parameter & $\tau$ & $\vartheta$ & $\lambda_{0}$ & $\tau^{G}$ & $\tau^{F}\vphantom{\sum_{k}^{L}}$ & $\vartheta^{G}$ & $\vartheta^{F}$ & $\alpha_{0}$ & $\gamma_{0}$\tabularnewline
	\hline 
	Value & $1.05$ & $0.7$ & $1$ & $1.05$ & $0.2$ & $0.7$ & $0.7$ & $1$ & $10^{7}$\tabularnewline
	\hline 
\end{tabular}
\caption{Hyperparameters for the parameter selection in the MPI experiment.}
\label{table:Hyper_parameters_MPI}
\end{table}
}

{\renewcommand{\arraystretch}{1.2}
\begin{table}
	\centering
	\begin{tabular}{|c|c||c|c|c|c|c|}
		\hline 
		\multicolumn{2}{|c||}{Tikhonov} &  & \multicolumn{2}{c|}{TSVD} & \multicolumn{2}{c|}{BPF method (Alg. 1)}\tabularnewline
		\hline 
		& (a) & kernel & (b) analytical & (c) approximated & (d) analytical & (e) approximated\tabularnewline
		\hline 
		$\lambda^{*}$ & $1.34\textup{e}-4$ & $\alpha^{*}$ & $0.12$ & $0.12$ & $0.24$ & $0.24$\tabularnewline
		\hline 
		&  & $\gamma^{*}$ & $33232$ & $3.81\textup{e}-10$ & $23263$ & $3.81\textup{e}-10$\tabularnewline
		\hline 
	\end{tabular}

	\caption{Regularisation parameters selected by Morozov's principle in case of the Tikhonov regularisation  and the double discrepancy principle for the double  regularisation for inverting the Radon transform and convolution.}
	\label{table:params_morozov_MPI}

\end{table}
}

For the next tests, we use a TSVD on the Radon transform and once deconvolve with a kernel derived analytically in \cite{BlankeBrandt_2023} (b) and once with the kernel obtained  by simulating a phantom consisting of a single voxel (c). In the last tests we regularise the Radon transform with the BPF-method sketched in Algorithm~\ref{alg:BFP} and regularise the convolution in the same way as described above, respectively (d) and (e). In Tests (b) to (e), we used the double discrepancy principle to determine the regularisation parameters. The hyperparameters for the double discrepancy principle can be found in Table~\ref{table:Hyper_parameters_MPI} and the selected regularisation parameters  can be found in Table~\ref{table:params_morozov_MPI}.

To get rid of the time-dependant factor $s(t)$, we applied Tikhonov regularisation with fixed regularisation parameter $0.05$ to the data for  tests (b)-(e), which amplified the Gaussian noise on the signal. With this amplified noise-level, we tested the double discrepancy principle with the respective regularisation methods. 

All reconstructions are visualised in Figure~\ref{fig:recosMPI}. For all images, only the positive part was plotted.

	\begin{figure}[h!]
\centering
\begin{subfigure}{\linewidth}
	\centering
	\includegraphics[]{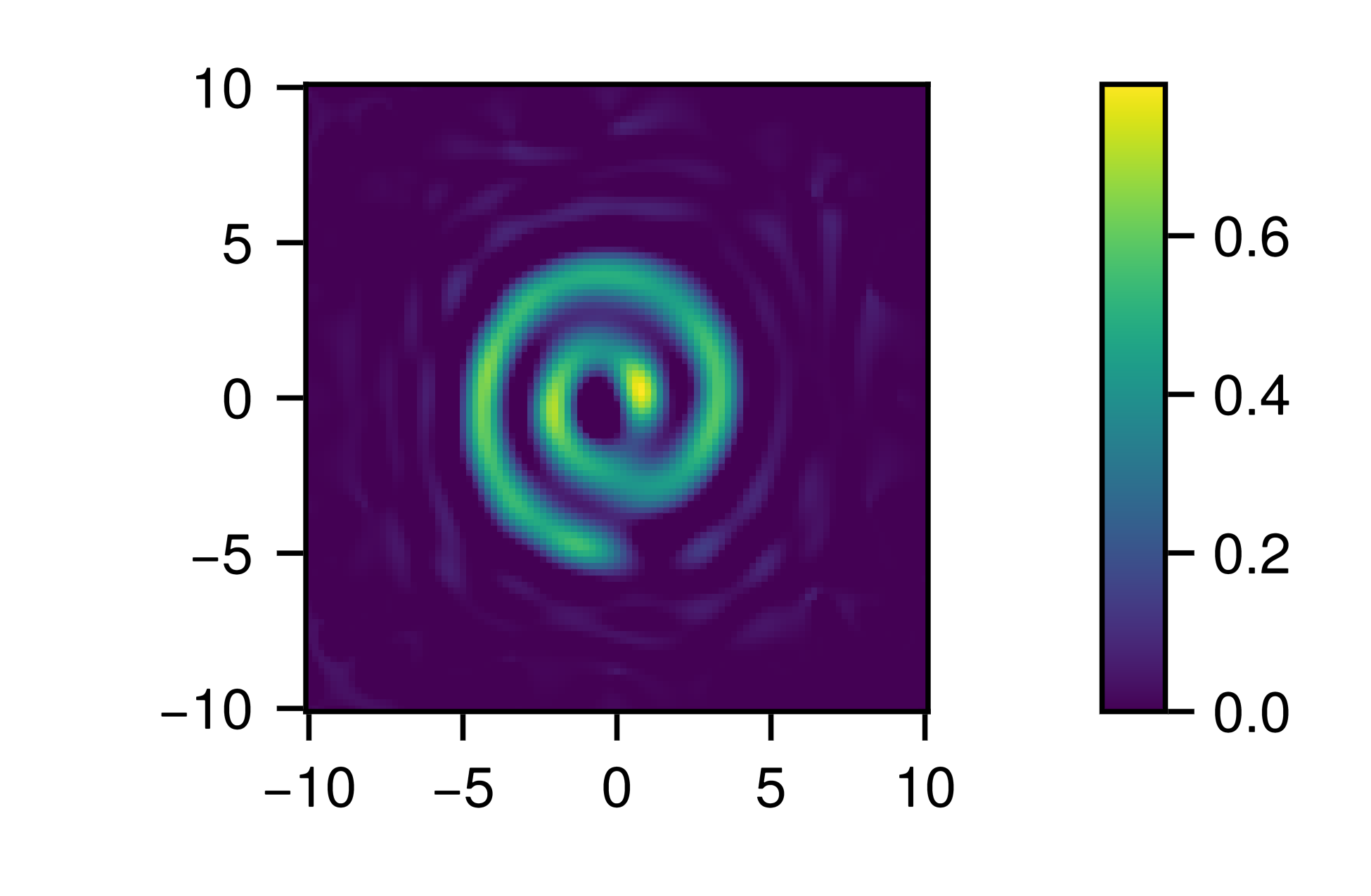}
	\subcaption{Reference reconstruction using 
		Kaczmarz algorithm and Morozov-discrepancy principle}
\end{subfigure}
\begin{subfigure}{0.45\linewidth}
	\centering
	\includegraphics[]{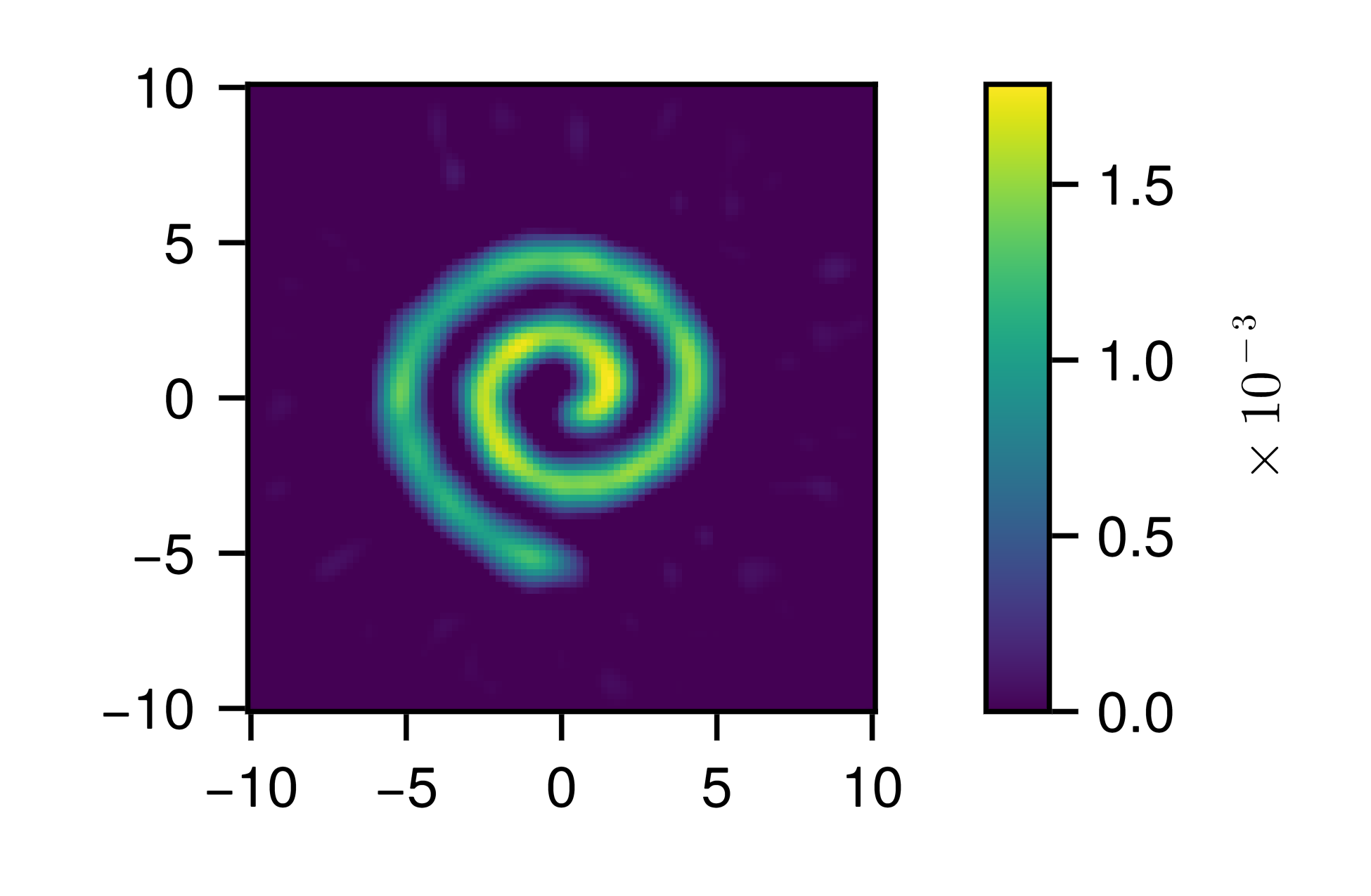}
	\subcaption{TSVD and analytical kernel}
\end{subfigure}\begin{subfigure}{0.45\linewidth}
	\centering
	\includegraphics[]{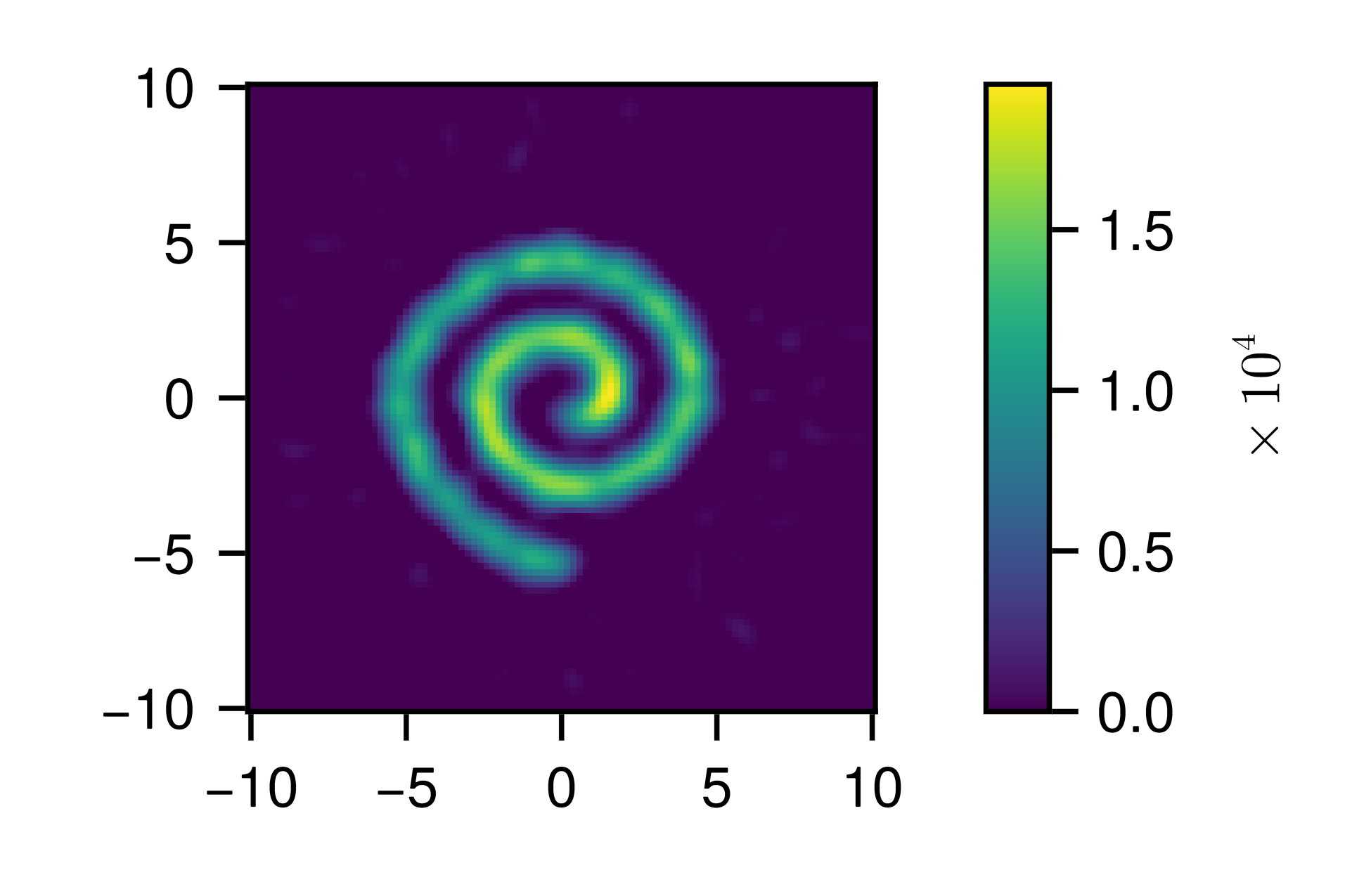}
	\subcaption{TSVD and approximated kernel}
\end{subfigure}
\begin{subfigure}{0.45\linewidth}
	\centering
	\includegraphics[]{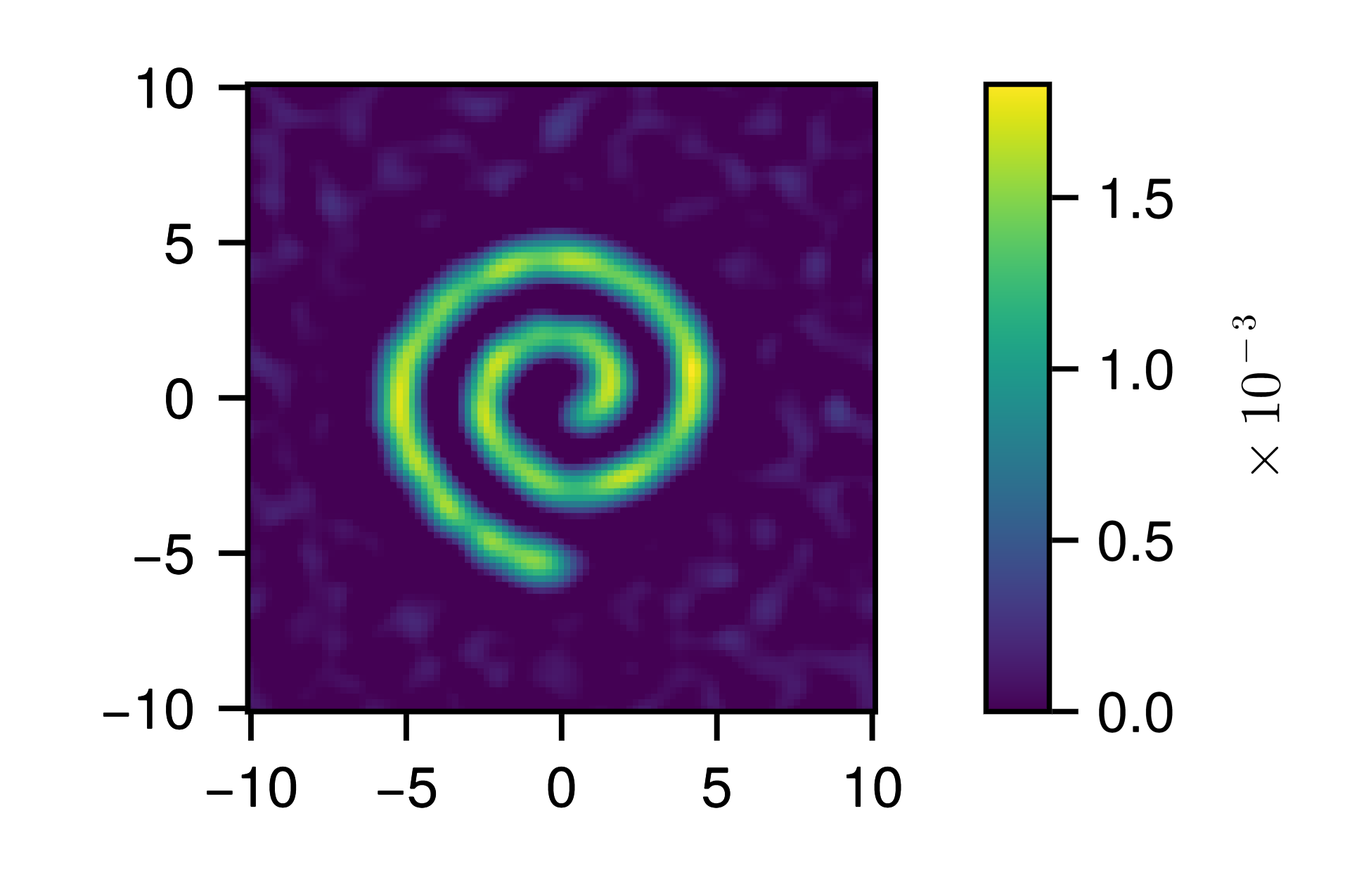}
	\subcaption{Alg~\ref{alg:BFP} and analytical kernel}
\end{subfigure}\begin{subfigure}{0.45\linewidth}
	\centering
	\includegraphics[]{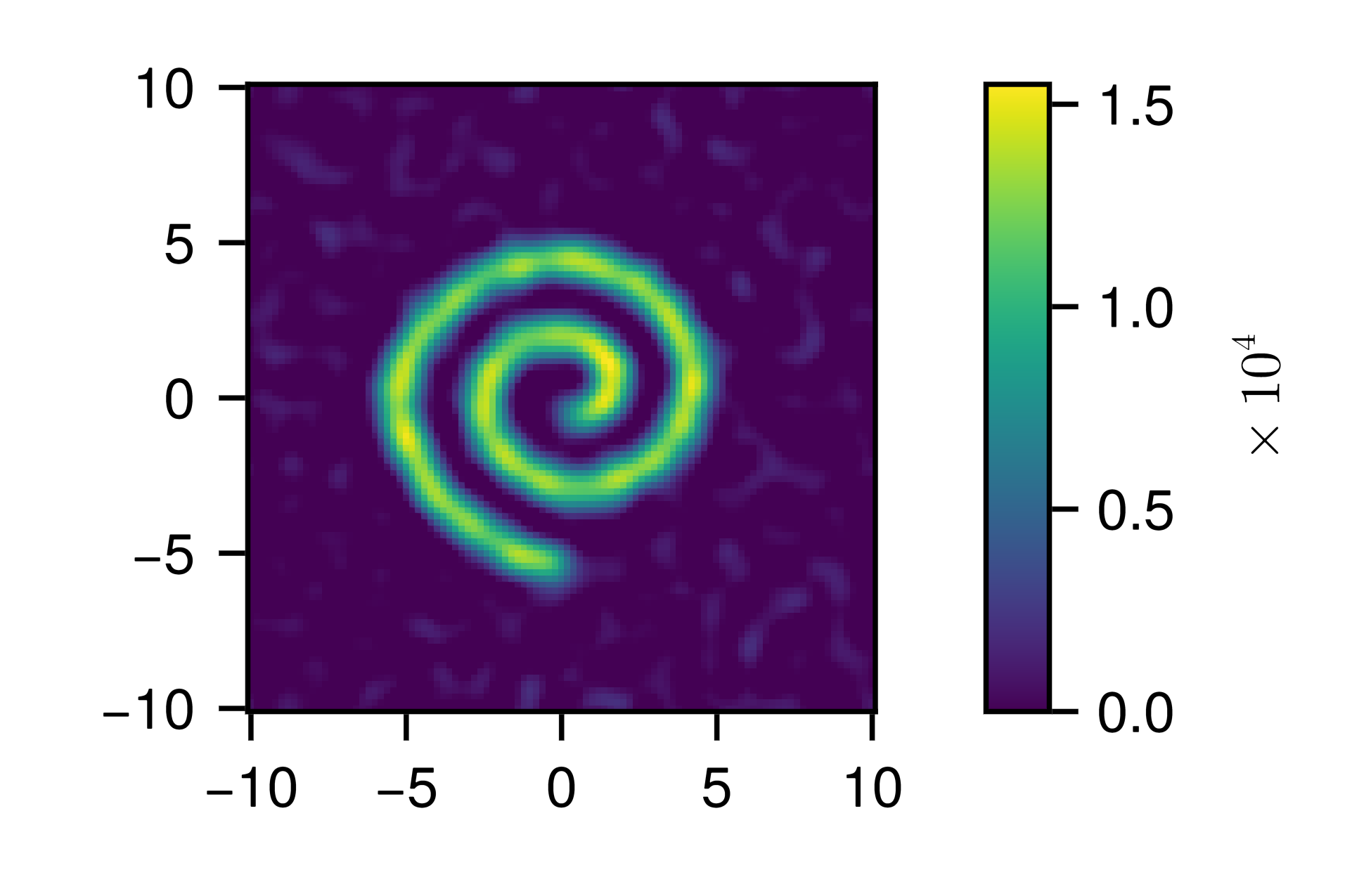}
	\subcaption{Alg~\ref{alg:BFP} and approximated kernel}
\end{subfigure}
\caption{Reconstructions of the spiral phantom. Reconstructions (b) to (e) are obtained by reconstruction from Radon data and deconvolution afterwards with the indicated kernel. The regularisation parameters are determined with help of  the double discrepancy principle, see Table~\ref{table:Hyper_parameters_MPI} and Table~\ref{table:params_morozov_MPI}.}
\label{fig:recosMPI}
\end{figure}

In all reconstructions, the general shape of the phantom is recognisable. In the reference reconstruction, there is a spot at the inner start of the spiral depicted with higher intensity which does not correspond to the desired phantom. In the other four reconstructions, the spiral is also not of uniform intensity, but there is no singular spot with higher intensity. Especially the reconstructions based on Algorithm~\ref{alg:BFP}, i.e. tests (d) and (e), show a more uniform intensity.

In the reference reconstruction, there are noise artefacts visible. In the reconstructions with the TSVD of the Radon transform, i.e. subfigures (b) and (c), there is almost no noise visible in the background. The background of the reconstructions (d) and (e) is less noisy than the background of the reference reconstruction, too. It seems as if the regularisation method for the Radon transform heavily influences the noise level of the background of the reconstruction.

The order of magnitude of the voxel intensities varies across the reconstructions. The different order of magnitudes in the reconstructions (b) and (d) to reconstructions (c) and (e) are caused by the different kernels. They also vary in order of magnitude but are very similar apart from this, which explains the different order of magnitude but rather similar images.

The difference in the order of magnitude of the reconstruction kernels cause a different order of magnitude in the selected regularisation parameters, too. Nevertheless, the regularisation parameters $\gamma^*$ are similar when using the same regularisation scheme even if they are used after different regularisations for operator $G^\dagger$.

All in all, the reconstructions with the double discrepancy principle show remarkable prowess, giving much better visual results than the benchmark method, both in the depiction of the shape of the phantom as well as the amount of background noise. This happens in a stable way, such that the parameter of the subsequent regularisation $R^F_\gamma$ apparently is not influenced much by the choice of regularisation $R^G_\alpha$.

The performance of the double discrepancy principle is especially notable in the context that the BPF-method is no regularisation via spectral filtering and thus does not fulfil the Assumptions~\ref{ass_filter}. This may indicate that the regularisation property can be shown with more general assumptions.

\subsection{Experiment 2: Convergence speed with the DDP}\label{sec:NumerikConv}

In this section, the objective is the evaluation of the properties of the double discrepancy principle presented in Section ~\ref{sec:convSpeed} given source conditions. In the context of the previously used MPI example, source conditions are unknown and consequently, the related problem of convolution is used as test problem. 
With help of the generalised functional calculus, see Appendix~\ref{sec:App:functCalc}, we derive source conditions for a test example. For this, we employ the multiplication operator $M_hf = t \mapsto h(t)f(t)$.
\begin{example}\label{exp:TextProbConv}
Consider now the two functions 
\begin{align*}
h_G(t) &= \frac{1}{1 + t^2},\\
h_F(t) &= \exp(-\abs{t}).
\end{align*}

and set $Ax = x * h_G * h_F$. Then 
$x = \abs{A}^\nu z$ if and only if
\[\mathcal F x = M_{\abs{\mathcal F h_G \mathcal F h_F}^\nu} \mathcal F z.\]
Therefore  
$x \in \mathcal R(\abs{A}^\nu)$ if and only if 
\[ M_{\abs{\mathcal F h_G \mathcal F h_F}^{-\nu}}\mathcal F x\in \mathrm L^2(\R).\]

With the true solution $x^\dagger(t) = \exp(-\frac{t^2}2)$, we obtain 
$$(\mathcal Fh_G \mathcal Fh_F)^{-\nu } \mathcal F x^\dagger (\xi)= \sqrt{2\pi}(2\pi)^\nu(1 + \xi^2)^\nu \exp\left(-\frac{\xi^2}2 + \nu \abs \xi\right) \in \mathrm L^2(\R)$$ for all $\nu \geq 0$, i.e. $x^\dagger \in X^A_\nu$ for all $\nu >0$.
\end{example}

We set the operators $A = G\circ F$ with $Gx = x * h_G$ and $Fx = x * h_F$.
For numerical tests, we discretise the operator using the fast Fourier transform and $N = 2^{14}$ uniformly spaced sampling points on the interval $(-100, 100)$ and added component-wise i.i.d. normally distributed noise.

For the convolution experiment, we conduct two tests. In a first test, we used two Tikhonov regularisations 
for $F^\dagger$ and $G^\dagger$.  
In the second test, we applied TSVD as regularisation each. 
The parameters for the \nameMethode ~were chosen according to Table~\ref{table:params_morozov}. 

In Figure~\ref{fig:ConvSpeedPlots} 
the behaviour of the error $\norm{x^\dagger - x^\delta}$ for $\delta \to 0$ 
is plotted for both experiments.
An exemplary reconstruction using TSVD regularisations and error $z^\dagger - z^\delta \sim \mathcal N(0, 0.0001 I)$, i.e. with standard deviation $\sigma = 0.01$, can be found in Figure~\ref{fig:recoConv}. This corresponds to an error level of around $\delta \approx 1.27$. The exact value of $\delta$ was subsequently used in the regularisation.

\begin{figure}
	\begin{subfigure}{0.45\linewidth}
	\centering
	\includegraphics[]{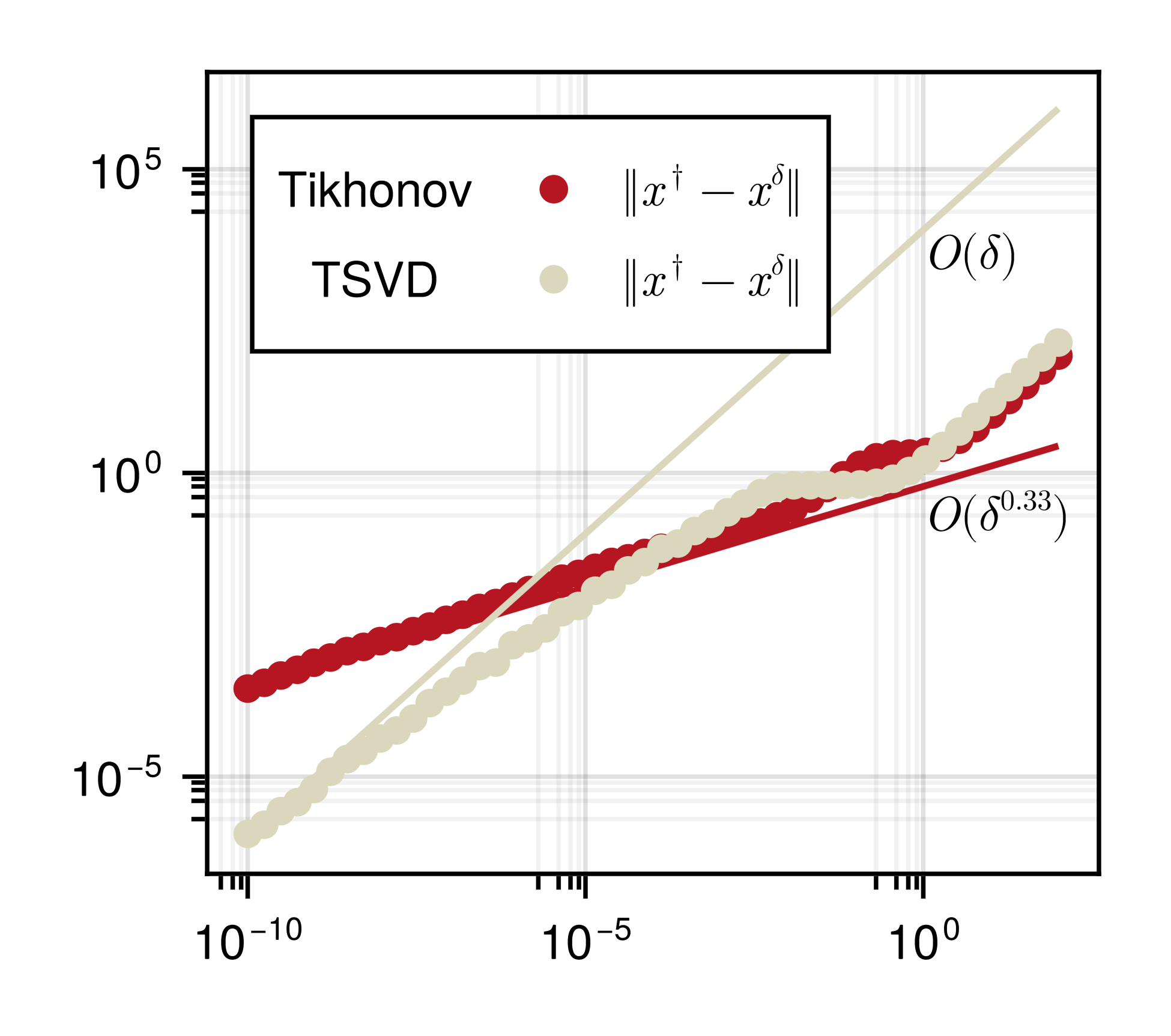}
	\caption{Convergence analysis of the \nameMethode}	\label{fig:ConvSpeedPlots}
\end{subfigure}\hfill\begin{subfigure}{0.45\linewidth}
\centering
	\includegraphics[]{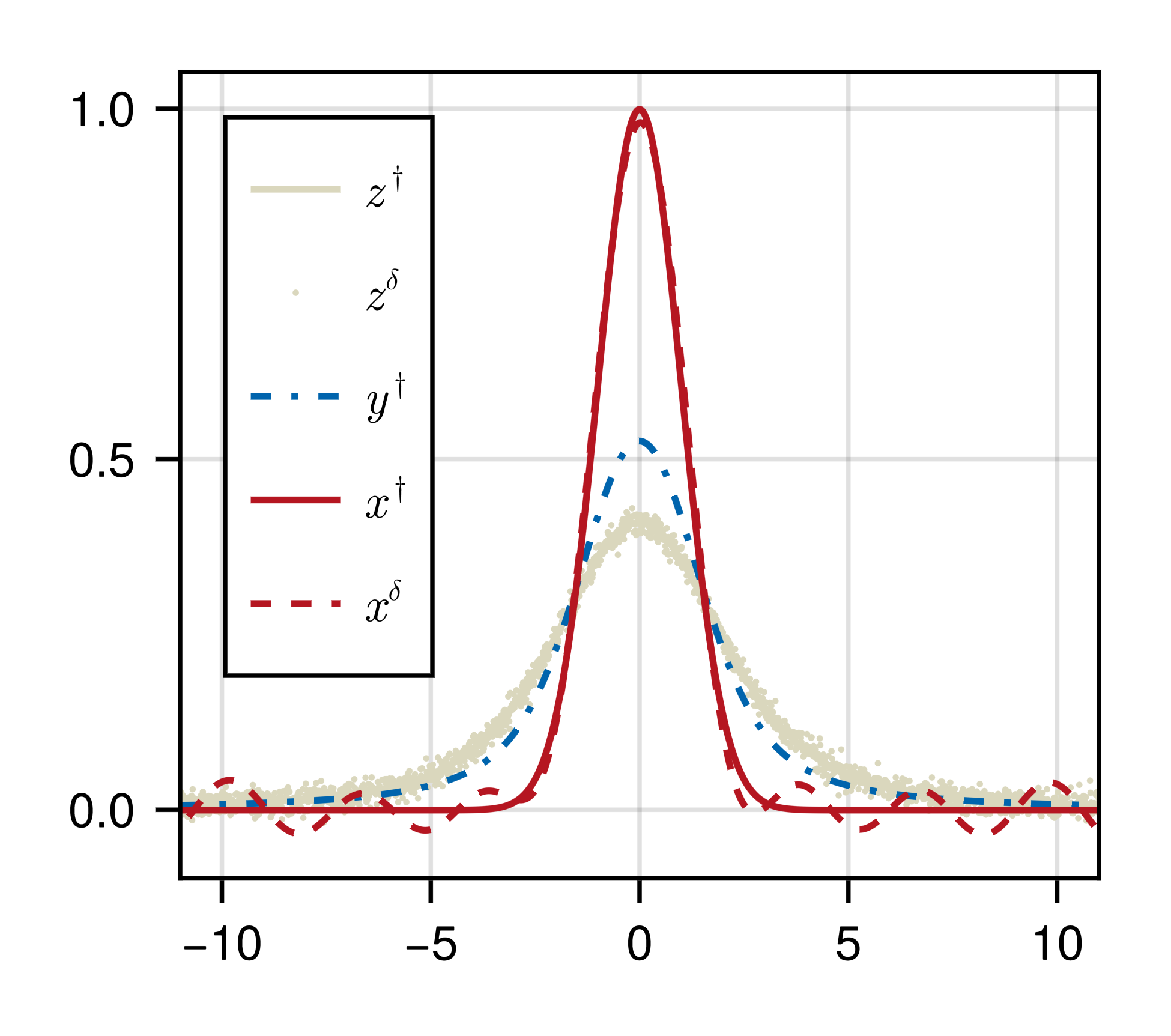}
	\caption{Exemplary reconstruction using the \nameMethode.}
	\label{fig:recoConv}
\end{subfigure}
\caption{Analysis of the DDP for the convolution experiment.}
\end{figure}

The reconstruction of $y$ from $z^\delta$ was successful, with the capture of the peak at $t=0$ and only minor oscillation occurring outside the peak. 
Nevertheless, the reconstruction from $y^\delta$ amplifies\ these oscillations along the flat portion of $x$. This observation may indicate that the regularisation parameter has been selected at a suboptimal level. 
However, the peak of $x^\delta$ is less than the peak of $x^\dagger$. Therefore, it can be deduced that the  the regularisation already dampens the details of $x^\dagger$, i.e. the regularisation parameter is chosen to be too high.

As illustrated in Figure~\ref{fig:ConvSpeedPlots}, the parameter choice rule results in convergence with both regularisations. In accordance with the established theoretical framework, the TSVD regularisation method, characterised by its infinite qualification of the filter, exhibited superior convergence properties in comparison to the Tikhonov regularisation approach, which does not possess this property. 
The numerical results indicate a convergence speed of $\mathcal O (\delta^{0.4})$ for Tikhonov regularisation and an order of $\mathcal O(\delta^{0.8})$ for the TSVD regularisation. According to the established result in Theorem~\ref{thm:RegMorozov}, the order is predicted to be $\delta^\frac{\nu_0 -1}{\nu_0 +1}$. In case of  the Tikhonov regularisation with $\nu_0 = 2$, we  thus expect a rate $\delta^\frac{1}{3}$ and in case of the TSVD an order of $\delta$. We indicated these theoretical results in Figure~\ref{fig:ConvSpeedPlots}, too.

In total, the DDP performs better than expected on the Tikhonov example and worse than expected on the TSVD example.

\begin{table}
	\centering

	{\renewcommand{\arraystretch}{1.4}
	\begin{tabular}{|l|c|c|c|}
		\hline 
		Parameter & $\tau^{G}=\tau^{F}$ & $\vartheta^{G}=\vartheta^{F}$ & $\alpha_{0}=\gamma_{0}$\tabularnewline
		\hline 
		Value & $1.5$ & $0.9$ & $1000$\tabularnewline
		\hline 
	\end{tabular}}

	\caption{Used hyperparameters of the \nameMethode~ for the convolution experiment.}
	\label{table:params_morozov}
\end{table}

\section{Conclusion}\label{sec:Conclusion}

In this paper, we proposed a method to design regularisation methods for structured inverse problems with applications in imaging and machine learning. By combining existing regularisations, it is possible to use optimised methods. We further derived a-priori parameter choice rules if source conditions are known and the double discrepancy principle, an a-posteriori choice rule if they are not known. All parameter choice rules were analysed with regard to convergence speed, resulting in the expected, but suboptimal result of $\mathcal O\left(\delta^{\left(\frac{\nu}{\nu+1}\right)^2}\right)$ for the a-priori choice rules. It remains an open question whether this suboptimal convergence speed is the optimal speed for this setting. Further, the double discrepancy principle is a generalisation of the discrepancy principle and has a suboptimal result of $\mathcal O\left( \delta^\frac{\nu-1}{\nu+1}\right)$ for regularisations with finite qualification. Next research angles include  the use of $\delta$-free parameter choice rules such as the $L$-curve method for this composite setting.

We tested the double discrepancy principle on a simulated MPI problem comparing the reconstructions with different regularisation methods to a benchmark regularisation using a system matrix and the Kaczmarz-algorithm. This gave visually great results, regardless of the regularisation methods and showed stability of the second chosen regularisation parameter independent of the regularisation method used in the step before. Moreover, the method performed well even for regularisation methods not covered in the Assumptions~\ref{ass_filter} used for the proof of the regularisation parameter, leading to future research angles. 

The convergence speed of the double discrepancy principle given source condition was explored
confirming the regularisation property derived theoretically and mostly confirming the convergence speed. In case of the Tikhonov regularisation with finite qualification, the method showed superior performance than expected, which needs further investigation.

\printbibliography
\newpage
\appendix
\section{Established Foundational Theory}
 
\subsection{Basics from Regularisation Theory}\label{subsec:RegThBasics}
The Picard-condition connects the source condition and the singular value decomposition of the operator in question.
\begin{thm}[Picard-Condition, \text{\cite[Prop 3.13]{EHN_1996}}]\label{thm:Picard}
Let $A : X \to Z$ linear and compact with singular value decomposition
$Ax = \sum _n \lambda _n\mu_n \SP{x}{x_n}z_n$ and $z \in \overline{\mathcal R(A)}$. Then $z\in \mathcal R(A)$ if and only if 
\[\sum_{n}\frac{\abs{\SP{z}{z_n}}^2}{\lambda_n^2\mu_n^2}<\infty.\]
\end{thm}

We apply the interpolation inequality in the proof of Theorem~\ref{thm:RegMorozov}.
\begin{thm}[Interpolation inequality \text{\cite[eqn. (2.49)]{EHN_1996}}]\label{eqn:IntEq}
For an operator $A: X \to Z$ and $x \in X$, the inequality
\begin{equation}
\norm{\abs{A}^rx} \leq \norm{\abs{A}^qx}^\frac r q\norm{x}^{1-\frac r q}
\end{equation}
applies for all $q > r > 0$.
\end{thm}
The following Lemma links the properties of the filter function with the norm of the related regularisation.
\begin{lem}[\text{\cite[Satz 3.3.3]{Rieder_2003}}]\label{lem:OpNormRF}
We have $\norm{R^F_s} \leq \sqrt{C_F \sup_{\lambda \in (0, \norm{F}^2]}\abs{f^F_s(\lambda)}} \leq \sqrt{C_FM_F s^{-1}}$.
\end{lem}

\subsection{Basics from Computer Tomography}

We present the Fourier Slice theorem connecting the Radon transform with the ground truth, which is paramount for proving reconstruction formulas.

\begin{thm}[Fourier-Slice Theorem \text{\cite{Natterer1986}}]\label{thm:FourierSlice}
	Let $f\in \mathscr S(\R^2)$. Then,
	\[\mathcal F_1 \mathcal Rf(s, \varphi) = \mathcal F_2 f(s\cos\varphi, s\sin\varphi).\]
\end{thm}

In order to use a regularisation for the Radon transform using filter methods, the singular value decomposition is needed. To state this, we need the weighted Hilbert space 
\[\left( \mathrm L^2( [-1, 1] \times \mathbb S^1), \norm \cdot _ {w^{-1}}\right)\]
with weight function $w(s, \varphi) = \sqrt{1-s^2}$, i.e. with the norm
\[\norm{f}_{w^{-1}}^2 = \int_{-1}^1 \int_{\mathbb S^1}\frac{\abs{f(s, \varphi)}^2}{w(s, \varphi)}\uprd \varphi \uprd s.\]
For this weighted Hilbert spaces, the singular value decomposition is explicitly known \cite{Davidson_1983, Louis_1984}.
\begin{thm}[\text{\cite[Satz 2.5.3]{Rieder_2003}}]\label{thm:SVDRadon}
The Radon transform $\mathcal R : \mathrm L^2(\Omega) \to \left( \mathrm L^2( [-1, 1] \times \mathbb S^1), \norm \cdot _ {w^{-1}}\right)$ on the domain $\Omega = \{(x,y) \in \R^2 \mid \norm{(x,y)^T}\leq 1\}$ has  the singular value decomposition 
\[\left\{ \sigma_{m, l}, v_{m, l}, u_{m, l} \mid m \in \N_0, l \in \Z, \abs l \leq m, m + l \in 2 \Z  \right\}\]
with
\begin{itemize}
\item singular value $\sigma_m = \sigma_{m,l} = 2 \sqrt{\frac{\pi}{m+1}}$ of multiplicity $m+1$,
\item the orthogonal basis of $\ker(\mathcal R)^\perp$, 
\[v_{m, l}(x,y) = \sqrt{\frac{m+1}{\pi}} \exp(i l \arg(x + iy)) \norm{(x,y)^T}^{\abs{l}} P^{(0, \abs{l})}_{\frac{m - \abs{l}}{2}}\left( 2 \norm{(x,y)^T}^2 -1\right),\]
\item and the orthogonal basis of $\mathrm{ran}(\mathcal R)$,
\[u_{m,l}(s, \varphi) = \frac 1 \pi \exp(i l \varphi) w(s) U_m(s).\]
\end{itemize}
We denoted by $U_m$ the Chebyshev polynomials of second kind  and by $P^{(\alpha, \beta)}_n$ the Jakobi polynomials .
\end{thm}

\section{Functional Calculus for Convolution }\label{sec:App:functCalc}
	In preparation for those numerical experiments, we fix an even, real function $h\in L^1(\R)\cap L^2(\R)$ and  consider the convolution operator
\begin{equation}\label{eqn:ConvOp}
		A: L^2(\R) \to L^2(\R),\quad x \mapsto x * h. 
\end{equation}

The convolution theorem implies
\[\mathcal F Ax = \mathcal F x \mathcal F                                                                                                                                                                                                                                                                                                                                                                                                                                                                                                                                                                                                                                                                                                                                                                                                                                                                                                                                                                                                                                                                                                                                                                                                                                                                                                                                                                                                                                                                                                                                                                                                                                                                                                                                                                                                                                                                                                                                                                                                                                                                                                                                                                                                                                                                                                                                                                                                                                                                                                                                                                                                                                                                                                                                                                                                                                                                                                                                                                                                                                                                                                                                                                                                                                                                                                                                                                                                                                                                                                                                                                                                                                                                                                                                                                                                                                                                                                                                                                                                                                                                                                                                                                                                                                                                                                                                                                                                                                                                                                                                                                                                                                                                                                                                                                                                                                                                                                                                                                                                                                                                                                                                                                                                                                                                       h.\]

To check for smoothness prerequisites, we need to define $\abs A$. As $A$ is non-compact, we restate functional calculus for non-compact operators. 
We denote by $\sigma(A)$ the spectrum of $A$ and by $B(K)$ the Banach space of measurable functions $K \to \R$ equipped with the supremum norm for some compact set $K$.
We denote $\mathrm{id}$ as the identity mapping and $\mathds{1}$ as the constant function $t \mapsto 1$.

\begin{thm}[\text{\cite[Satz VII.1.6] {Werner_2018}}]\label{thm:FunCalc}
	Let $H$ be a Hilbert space and $A \in \mathrm{L} (H)$ be self-adjoint. Then there exists a unique mapping
	\[\Phi: B(\sigma(A)) \to \mathrm L(H)\] with the following properties:
	\begin{enumerate}
		\item $\Phi(\mathrm{id}) = A$ and $\Phi(\mathds 1) = \mathrm{id}$,
		\item $\Phi$ is an involutary algebrahomomorphism, i.e.
		\begin{enumerate}
			\item $\Phi$ is linear,
			\item $\Phi$ is multiplicative $\Phi(f \cdot g) = \Phi(f) \circ \Phi(g)$,
			\item $\Phi$ is involutary $\Phi\left(\bar f\right) = \Phi(f)^*$,
		\end{enumerate}
		\item $\Phi$ is continuous,
		\item For any family $\left\{ f_n\right\} \in B(\sigma(A))$ with $\sup_n \norm{f_n} <\infty$ converging point-wise to $f \in B(\sigma(A))$ we have for all $x, y \in H$
		\[\SP{\Phi(f_n)x}{y} \to \SP{\Phi(f)x}{y}.\]
	\end{enumerate}
We set 
\begin{equation}
f(A) := \Phi(f)
\end{equation}
 and call $\Phi$ the functional calculus of functional $A$.
\end{thm}

The spectrum of the convolution operator \eqref{eqn:ConvOp}  
is given by $\sigma(A) = \mathcal F(h)(\R)\subset \R$. 
Using again the multiplication operator $M$, the convolution operator has the representation
\begin{equation}\label{eqn:AfunCalc}
	A = \mathcal F^{-1} M_{\mathcal Fh}\mathcal F.
\end{equation}

The following theorem formulates the functional calculus of $A$.
\begin{thm}
	For $f \in B(\sigma(A))$, it holds
	\[f(A) = \mathcal F^{-1} M_{f\circ\mathcal Fh}\mathcal F.\]
	\begin{proof}
		We need to show that the conditions of Theorem~\ref{thm:FunCalc} are fulfilled.
		\begin{enumerate}
			\item It follows immediately from representation~\eqref{eqn:AfunCalc}.
			\item \begin{enumerate}
				\item The identity  
				$M_{\alpha f + \beta g} = \alpha M_f + \beta M_g$ 
				implies
				\[(\alpha f + \beta g)(A) = \mathcal F^{-1} (\alpha M_{f \circ \mathcal Fh} + \beta M_{g \circ \mathcal Fh})\mathcal F.\]
				The linearity follows from the linearity of $\mathcal F$.
				\item We have
				\[ 	(f \cdot g)(A) = \mathcal F^{-1}M_{(f\circ \mathcal Fh) \cdot (g \circ \mathcal Fh)}\mathcal F= \mathcal F^{-1}M_{f\circ \mathcal Fh}M_{g \circ \mathcal Fh}\mathcal F= f(A) \circ g(A)\,.\]
			
				\item From $M_h^* = M_{\bar h}$ it follows that
				\[					f(A)^* = (\mathcal F)^* M_{f\circ \mathcal F h}^* (\mathcal F^{-1})^*=\mathcal F^{-1} M_{\bar f \circ \mathcal Fh} \mathcal F= \bar f (A).\]
		
			\end{enumerate}
			\item Let $f \in B(\sigma(A))$. Then we obtain
			\[
				\norm{f(A)} = \sup_{\norm{x} \leq 1}\norm{f(A)x}_2=\sup_{\norm{x} \leq 1}\norm{\mathcal F^{-1}  M_{f \circ \mathcal Fh}\mathcal Fx}_2= \sup_{\norm{\tilde x}\leq 1}\norm{M_{f \circ \mathcal Fh}\tilde x}_2,
			\]
		
			as the Fourier transform is, up to a constant, an isometric isomorphism. Moreover, it holds
			\begin{align*}
				\norm{M_{f \circ \mathcal Fh}x}_2 &\leq \norm{f\circ\mathcal Fh}_\infty \norm{x }_2 \leq \norm{f}_\infty \norm{x}_2.
			\end{align*}
			\item Let $\left\{f_n\right\} $ be a family as desired in the theorem with $\sup_n\norm{f_n}_\infty<N$ and $f$ be the point-wise limit. For
			all $n\in \N$ and $x, y \in L^2(\R)$ it holds that 
			\[f_n(\mathcal Fh(t))x(t)\overline{y(t)}\in L^1(\R)\]
			and each of these functions is bounded by $N\abs{x(t)y(t)}\in L^1(\R)$. Thus,  the Lebesgue convergence theorem implies
			\begin{multline*}
				\SP{f_n(A)x}{y} = \SP{M_{f_n \circ \mathcal Fh}\mathcal Fx }{\mathcal F y}
				= \int_\R f_n(\mathcal Fh(t))\hat x(t) \overline{\hat y(t)}\, \mathrm dt\\
				\to \int_\R f(\mathcal Fh(t))\hat x(t) \overline{\hat y(t)}\, \mathrm dt
				= \SP{M_{f \circ \mathcal Fh}\mathcal Fx }{\mathcal F y}= \SP{f(A)x}{y}.
			\end{multline*}
			
			Here, we abbreviated $\mathcal F x = \hat x$ and $\mathcal Fy = \hat y$.
		\end{enumerate} 
	\end{proof}
\end{thm}

\subsection*{Acknowledgments}
We thank Dr. Patrick Vogel and the Magnetic Particle Imaging research group at the University of Würzburg who provided the MPI simulation software MFS5 \cite{SoftwareVogel} and offered technical support.

\end{document}